\documentclass[journal]{IEEEtran}

\usepackage[utf8]{inputenc} 
\usepackage{amsmath} 
\usepackage{amssymb} 
\usepackage[algo2e]{algorithm2e} 
\usepackage{algorithm} 
\usepackage{siunitx} 

\usepackage{amsthm} 
\newtheorem{theorem}{Theorem} 
\newtheorem{assume}{Assumption}
\newtheorem{lemma}{Lemma}

\usepackage{graphicx}
\usepackage[caption=false,font=footnotesize]{subfig}
\graphicspath{ {images/} }

\usepackage{mathtools}

\usepackage{tikz} 
\usetikzlibrary{patterns,arrows,calc,decorations.pathmorphing}

\begin{document}

\title{Two-Stage Dual Dynamic Programming with Application to Nonlinear Hydro Scheduling}

\author{Benjamin~Flamm,
        Annika~Eichler,~\IEEEmembership{Member,~IEEE},
        Joseph~Warrington,~\IEEEmembership{Member,~IEEE},
        and~John~Lygeros,~\IEEEmembership{Fellow,~IEEE} 
\thanks{The authors are with the Automatic Control Laboratory of ETH Zürich, Physikstrasse 30, 8092 Zürich, Switzerland. {\tt\small \{flammb, eichlean, warrington, lygeros\}@control.ee.ethz.ch}}
}

%

\maketitle

\begin{abstract}
We present an approximate method for solving nonlinear control problems over long time horizons, in which the full nonlinear model is preserved over an initial part of the horizon, while the remainder of the horizon is modeled using a linear relaxation. As this approximate problem may still be too large to solve directly, we present a Benders decomposition-based solution algorithm that iterates between solving the nonlinear and linear parts of the horizon. This extends the Dual Dynamic Programming approach commonly employed for optimization of linearized hydro power systems. We prove that the proposed algorithm converges after a finite number of iterations, even when the nonlinear initial stage problems are solved inexactly. We also bound the suboptimality of the split-horizon method with respect to the original nonlinear problem, in terms of the properties of a map between the linear and nonlinear state-input trajectories. We then apply this method to a case study concerning a multiple reservoir hydro system, approximating the nonlinear head effects in the second stage using McCormick envelopes. We demonstrate that near-optimal solutions can be obtained in a shrinking horizon setting when the full nonlinear model is used for only a short initial section of the horizon. For this example, the approach is shown to be more practical than both conventional dynamic programming and a multi-cell McCormick envelope approximation from literature. 
\end{abstract}


%
\IEEEpeerreviewmaketitle

\section{Introduction}
%
%
%
\IEEEPARstart{T}{he} control of energy storage devices over a long planning horizon is gaining in importance, as the number of renewable and variable energy sources on the electric power grid increases. U.S. installed wind and solar generation capacity increased by 11\% and 52\% respectively in 2016 \cite{USDoe2016}. Energy storage is particularly well-suited to complement the resulting unpredictable power generation. Long-term energy storage exists in different forms, with varying technical maturity and efficiency \cite{Chen2009}. Examples include natural gas for heating, compressed air electrical storage, ground thermal energy storage, and hydro reservoir systems; the latter have been extensively utilized due to their large storage capacity, technological maturity, and proliferation. Key challenges for the integration of seasonal storage into next generation energy systems are the long horizons and nonlinear system dynamics, which can render long-term control of the storage computationally difficult. 

We typically wish to operate long-term storage devices in response to underlying energy demand or supply patterns that are cyclical over a period of months or years. Examples of such patterns include yearly snow melt, as well as seasonal heating and cooling demands. 

Several methods have been proposed to tackle nonlinear seasonal storage problems. Many involve approximating the nonlinear dynamics through modeling simplifications and heuristic methods such as timescale separation \cite{Abgottspon2016}. Specifically for hydro optimization, a fundamental difficulty is the presence of nonlinear head effects when converting between stored water and electrical energy. Many papers ignore nonlinear head effects and represent energy conversion as a constant efficiency \cite{Pritchard2005}, \cite{Rotting1992}. Methods that do account for head effects usually generate convex hulls of the power production function, either by fitting a set of piecewise linear constraints to the true model \cite{Diniz2008}, \cite{Borghetti2008}, or by rewriting nonlinear terms using convex approximations such as McCormick envelopes \cite{Cerisola2012}. 

Other methods take advantage of improvements in computing power, which allows increasingly large nonlinear control problems to be solved exactly. Approaches that use nonlinear models usually consider only short horizons \cite{Catalao2009}. Techniques such as spatial branch-and-bound and dynamic programming allow for the solution of nonlinear problems to high precision, but scale poorly with the state space size. 

Dual dynamic programming (DDP), also known as multistage Benders decomposition, was introduced as a method for seasonal hydro storage scheduling in \cite{Pereira1991}, but has since been used primarily for solving linear approximations of these problems. DDP was extended to convex problems using generalized Benders decomposition in \cite{Grossmann1991}. Roughly speaking, DDP bounds the value function of a convex problem by a piecewise affine function of the initial state. Since DDP cannot solve problems with nonconvex value functions, convex approximations of the model are required when using DDP for nonlinear problems. Recent extensions of this approach have considered integer programs \cite{Zou2018}, locally-valid Benders cuts \cite{AbgottsponThesis2015}, and polynomial-based moment relaxation \cite{Hohmann2018}.

In \cite{Flamm2018}, a method was introduced to treat general receding horizon nonlinear optimal control problems. The full problem horizon was split into short- and long-term parts, with high model accuracy in the short term, and reduced model accuracy in the long term. Here we expand this approach by introducing several convergence and optimality results related to solving this approximation of the underlying nonlinear problem. Unlike \cite{Flamm2018}, in Theorem \ref{thm:local_solution_ddp_convergence} we provide a guarantee on convergence of the proposed DDP algorithm, even when the nonlinear first-stage problem is solved suboptimally. Additionally, inspired by \cite{Guigues2018}, where linear subproblems were solved to a known tolerance in DDP, in Theorems \ref{thm:two_stage_error_bound} and \ref{theorem:two_stage_argument_optimality_bound} we provide bounds on the suboptimality of the solution of the two-stage approximation. These bounds are a function of the first-stage nonlinear solver tolerance as well as the approximation error between linear and nonlinear models in the second stage. We also demonstrate the generality of the approach by applying it to a different kind of nonlinearity than in \cite{Flamm2018}, involving bilinear as opposed to integer terms. Finally, we quantify how the control horizon (the decision length used each time the model predictive control (MPC) problem is solved) and modeling accuracy affect optimality in simulation.

The underlying stochasticity of real-time  optimal control problems is an additional source of computational complexity, and lends itself to the use of DDP. This can be treated by introducing scenarios to capture different potential realizations of the stochastic processes \cite{Rebennack2016}. With additional variables introduced for different scenarios, long-term problems with even simple models can push the limits of tractability. While much of the hydro optimization literature incorporates stochasticity, we consider a deterministic setting here to focus on the underlying computational method and treatment of nonlinearity.

This paper is structured as follows. Section II introduces the general case of the split-horizon approximate problem, as well as a DDP-based algorithm to solve it. The optimality of the solution produced by the algorithm is derived as a function of the optimality of the first-stage solution. In Section III, a bound is found for the error of the two-stage approximation relative to the exact problem. Section IV formulates the nonlinear hydro optimal control problem, providing specific models and approximations of the underlying hydro system. In Section V, we analytically determine the error bound on McCormick envelope approximations of bilinear terms occurring in the problem. Section VI presents simulation results for a representative hydro system, illustrating the computational benefits of the multistage method. Section VII concludes with analysis of possible improvements and extensions to the proposed method.

\section{Split-Horizon Problem and DDP Algorithm}

\subsection{General Problem Formulation}
We wish to solve a generic discrete-time optimal control problem for a system with dynamics $f_t$, input and state constraints defined by the set $\mathcal{Z}_t$, and cost function $g_t$ over a horizon of length $T$. The state trajectory is denoted by $(x_{0},\ldots,x_{T})$, while the input trajectory is $(u_0,\ldots,u_{T-1})$, where $x_{t} \in \mathbb{R}^{n_t}$ and $u_t \in \mathbb{R}^{m_t}$. The optimal inputs are determined by solving the nonlinear program (NLP)
\begin{subequations} \label{eq:original_minlp}
\begin{align} 
\min_{\substack{x_{1},\ldots,x_{T} \\ u_{0},\ldots,u_{T-1} }} \
& \sum_{t=0}^{T-1} g_t(x_t,u_t) + g_T(x_{T}) 
\\ \label{eq:original_minlp_dynamics}
\text{s.t.} \quad
&  x_{t+1} = f_t(x_t,u_t), & t = 0, \ldots, T-1
\\ \label{eq:original_minlp_set_constraint}
&  (x_t,u_t) \in \mathcal{Z}_t, & t = 0, \ldots, T-1
\\
& x_0 \text{ given,}
\end{align}
\end{subequations}
where $f_t(x_t,u_t)$, $g_t(x_t,u_t)$, and $g_T(x_{T})$ are general nonlinear functions on the set $ \mathcal{Z}_t$. The dynamics, costs, and constraints can all be time-varying. This formulation is general in terms of variable type. For example, integer variables can be considered by adding constraints to $Z_t$.

Note that this is a quite general nonlinear programming problem whose complexity grows rapidly in the dimension of the state and input, as well as the horizon length.

\subsection{Two-Stage Approximation}

In \cite{Flamm2018}, to solve \eqref{eq:original_minlp} over a long horizon, the authors proposed solving an approximate problem consisting of a short initial stage of length $T_1$ where the exact NLP holds, and a subsequent longer stage of length $T-T_1$ where the problem is approximated by a linear program (LP):
\begin{subequations} \label{eq:original_two_stage_problem}
\begin{align} 
\min_{\substack{x_1,\ldots,x_{T_1} \\ u_0,\ldots,u_{T_1-1} \\  \tilde{x}_{T_1+1},\ldots,\tilde{x}_T  \\ \tilde{u}_{T_1},\ldots,\tilde{u}_{T-1}}  } \
& \sum_{t=0}^{T_1 - 1} g_t(x_t,u_t) + \sum_{t=T_1}^{T-1} \! ( c_t^{\top} \tilde{x}_t + d_t^{\top} \tilde{u}_t ) + c_T^\top \tilde{x}_{T}  
\\ \label{eq:nonlinear_first_stage_constraints}
\text{s.t.} \quad & \begin{rcases} x_{t+1} = f_t(x_t,u_t), \quad \
\\ 
(x_t,u_t) \in \mathcal{Z}_t,
\end{rcases} \quad t = 0, \ldots, T_1-1
\\ 
& \begin{rcases}
\tilde{x}_{t+1} = A_t \tilde{x}_t + B_t \tilde{u}_t, 
\\ \label{eq:linear_second_stage_constraints}
E_t \tilde{x}_t + F_t \tilde{u}_t \leq h_t,
\end{rcases} \quad t = T_1, \ldots, T-1
\\ \label{eq:intermediate_state_coupling}
& x_{T_1} = \tilde{x}_{T_1},
\\
& x_{0} \text{ given.}
\end{align}
\end{subequations}

Here, the first stage of the problem, from $t = 0$ to $T_1-1$, has the exact nonlinear cost function, dynamics, and constraints. In the second stage of the problem, these are replaced with linear costs, dynamics, and constraints. 

Note that the state of the approximate linear part at time $t$ is denoted by $\tilde{x}_t$ (instead of $x_t$), to stress that the second-stage approximation is different from the original problem, potentially incorporating different state variables at each timestep (for example, a relaxation of integer variables in the original problem to real-valued variables in the linear part). However, by \eqref{eq:intermediate_state_coupling}, the states $x_t$ and $\tilde{x}_t$ coincide at $t = T_1$.

We now recast the nonlinear approximate problem \eqref{eq:original_two_stage_problem} as an equivalent two-stage problem. The nonlinear first stage can be written as
\begin{subequations} \label{eq:first_stage_problem_with_terminal_cost}
\begin{align} \label{eq:first_stage_objective}
\min_{\substack{x_1,\ldots,x_{T_1} \\ u_0,\ldots,u_{T_1-1} }} \
&  \sum_{t=0}^{T_1 - 1} g_t(x_t,u_t) + \tilde{G}_{T_1}(x_{T_1}) 
\\
\text{s.t.} \quad
&  x_{t+1} = f_t(x_t,u_t), \quad t = 0, \ldots, T_1-1
\\ \label{eq:first_stage_feasible_set}
&  (x_t,u_t) \in \mathcal{Z}_t, \quad \quad \ \, \  t = 0, \ldots, T_1-1
\\
&x_{0} \text{ given.}
\end{align}
\end{subequations}

In \eqref{eq:first_stage_objective}, $\tilde{G}_{T_1}(x_{T_1})$ is a value function that represents the cost of the linear second stage of \eqref{eq:original_two_stage_problem} as a function of the state at the end of the first stage, with
\allowdisplaybreaks
\begin{subequations} \label{eq:second_stage_problem}
\begin{align} 
\tilde{G}_{T_1}(x_{T_1}) = \min_{\substack{\tilde{x}_{T_1+1},\ldots,\tilde{x}_{T} \\ \tilde{u}_{T_1},\ldots,\tilde{u}_{T-1} }} \
& \sum_{t=T_1}^{T-1} \left( c_{t}^{\top} \tilde{x}_t + d_t^{\top} \tilde{u}_t \right) + c_T^\top \tilde{x}_{T}
\\ \label{eq:second_stage_equality}
\text{s.t.} \quad
& \tilde{x}_{t+1} = A_t \tilde{x}_t + B_t \tilde{u}_t,
\\ \label{eq:second_stage_inequality}
& E_t \tilde{x}_t + F_t \tilde{u}_t \leq h_t,
\\ \nonumber
& \quad \ t = T_1, \ldots, T-1,
\\ 
& \tilde{x}_{T_1} = x_{T_1} \text{ given.}
\end{align}
\end{subequations}

The division of the horizon into the two stages and intermediate value function is depicted in Figure~\ref{fig:ddp_method_illustration}.

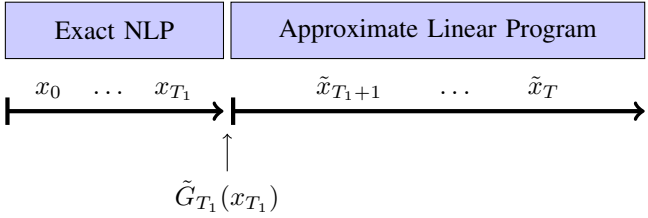
\begin{figure} 
	\begin{centering}
	\begin{tikzpicture}
		\filldraw[fill=blue!20, draw=black] (0,0) rectangle (2.9,0.75) node[pos=.5] {Exact NLP};
  		\filldraw[fill=blue!20, draw=black] (3.0,0) rectangle (8.5,0.75) node[pos=.5] {Approximate Linear Program};
    
  		\draw[|->,ultra thick] (0,-0.7)--(2.9,-0.7) node[above,pos=.5]{$x_{0} \quad \ldots \quad x_{T_1}$};
  
  		\draw[|->,ultra thick] (3.0,-0.7)--(8.5,-0.7) node[above,pos=.5]{$\tilde{x}_{T_1+1} \quad \quad \ldots \quad \quad \tilde{x}_{T}$};
  
  		\draw[->] (2.95,-1.5)--(2.95,-1) node[below,pos=0]{$\tilde{G}_{T_1}(x_{T_1})$};
	\end{tikzpicture}
	
	\caption{Diagram of split-horizon method, with exact short-term nonlinear problem (NLP) and intermediate value function representing the approximate linear long-term problem.}
	\label{fig:ddp_method_illustration}
	\end{centering}
\end{figure}

To ensure that the subsequent theory and analysis is meaningful, we assume that both \eqref{eq:original_minlp} and \eqref{eq:original_two_stage_problem} are feasible. We additionally make the following assumption, which facilitates the iterative solution of the two-stage problem.
\begin{assume} \label{assumption:second_stage_feasibility}
The linear second-stage problem \eqref{eq:second_stage_problem} is feasible for all $x_{T_1}$ that are feasible for the first-stage problem \eqref{eq:first_stage_problem_with_terminal_cost}.
\end{assume}
This assumption, known in the mathematical programming literature as \textit{complete recourse}, is not strictly necessary, as the theory related to Benders decomposition used here also works for problems where the second stage is not feasible for a particular $x_{T_1}$. In that case, Benders feasibility cuts on $x_{T_1}$ can be iteratively added to the first-stage problem \eqref{eq:first_stage_problem_with_terminal_cost}, as is done in \cite{Grossmann1991}. However, for simplicity we omit this case here.

\subsection{Split-Horizon Approximate DDP Algorithm}
\begin{algorithm} 
\caption{Split-horizon DDP with approximate solution to nonlinear first stage} \BlankLine
 
{Define:} \parbox[t]{7cm}{$T_1$: length of nonlinear first stage\,;
\\ $x_0$: starting system state}

\BlankLine
{Initialize:} \parbox[t]{7cm}{$\mathcal{H} = \emptyset$\,;
\\ $UB = +\infty$, $LB = -\infty$\,;}

\BlankLine
1. Solve first-stage NLP \eqref{eq:first_stage_problem_with_terminal_cost}, with $\tilde{G}_{T_1}(x_{T_1}) = 0$: $\left((u_0,\ldots,u_{T_1-1}),(x_1,\ldots,x_{T_1}) \right)$ $\gets$ feasible solution to \eqref{eq:first_stage_problem_with_terminal_cost}\;;\BlankLine

\While{(true)}{ 
2. Solve second-stage linear program \eqref{eq:second_stage_problem} with argument $x_{T_1}$ as found in previous step. Returned arguments are $(\tilde{u}_{T_1},\ldots,\tilde{u}_{T-1})$ and $(\tilde{x}_{T_1+1}, \ldots, \tilde{x}_T)$. Let dual variables $\lambda$ and $\nu$ correspond to \eqref{eq:second_stage_equality} and \eqref{eq:second_stage_inequality} at $t=T_1$\;;\BlankLine
3. Set $UB = \sum_{t=T_1}^{T-1} \left( c_t^{\top} \tilde{x}_t + d_t^{\top} \tilde{u}_t \right) + c_T^\top \tilde{x}_T$\,;\BlankLine
4. \If{$UB = LB$}{\textbf{break\,;}} \BlankLine
5. Let $a = A^\top_{T_1} \lambda + E^\top_{T_1} \nu + c_{T_1}$ and $b = - \nu^\top h_{T_1}$, and update collection of hyperplanes bounding $\tilde{G}_{T_1}(x_{T_1})$ as $\mathcal{H} = \mathcal{H} \cup \{(a,b)\}$\,;\BlankLine
6. Solve first-stage NLP \eqref{eq:first_stage_problem_with_terminal_cost}, replacing $\tilde{G}_{T_1}(x_{T_1})$ by $\underline{G}_{T_1}(x_{T_1}) = \max_{(a,b) \in \mathcal{H}}(a^\top x_{T_1} + b)$: $\left((u_0,\ldots,u_{T_1-1}),(x_1,\ldots,x_{T_1}) \right)$ $\gets$ feasible solution to \eqref{eq:first_stage_problem_with_terminal_cost}\;;\BlankLine
7. Set $LB = \underline{G}_{T_1}(x_{T_1})$\,;\BlankLine
}
\BlankLine
{Output:} \parbox[t]{7cm}{$u_\text{full} = (u_0,\ldots,u_{T_1-1},\tilde{u}_{T_1},\ldots,\tilde{u}_{T-1})$\,;
\\ $x_\text{full} = (x_1,\ldots,x_{T_1},\tilde{x}_{T_1+1},\ldots,\tilde{x}_T)$\,;}
\label{alg:mpc_ddp_algorithm}
\end{algorithm}

In \cite{Flamm2018}, we presented an algorithm that converged in a finite number of iterations to an optimum of the approximate two-stage problem \eqref{eq:original_two_stage_problem}, provided the first-stage NLP \eqref{eq:first_stage_problem_with_terminal_cost} was solved to global optimality. In certain cases, it may not be desirable or possible to solve \eqref{eq:first_stage_problem_with_terminal_cost} to global optimality, due to an algorithmic choice of solution tolerance \cite{Guigues2018}, \cite{Rahmaniani2017}, or constraints on available solution time. 

Algorithm \ref{alg:mpc_ddp_algorithm} presents an adaptation of this setting, where the first-stage NLP \eqref{eq:first_stage_problem_with_terminal_cost} is not solved to global optimality. Before analyzing the convergence of this algorithm, as described in Theorem \ref{thm:local_solution_ddp_convergence}, we first explain the value function approximation used therein. We wish to solve a two-stage problem \eqref{eq:original_two_stage_problem} using DDP. To do this, we iteratively solve the two stages of the problem, using DDP to progressively construct a set of lower-bounding hyperplanes for the second-stage value function $\tilde{G}_{T_1}(x_{T_1})$. The lower-bounding hyperplanes are found using the following lemma.

\begin{lemma} \label{lemma:ddp_lb}
(Lemma 1 in \cite{Flamm2018}) Given a solution to \eqref{eq:second_stage_problem} for some $x_{T_1}$, let $\lambda$ and $\nu$ be the dual variables corresponding to the constraints \eqref{eq:second_stage_equality} and \eqref{eq:second_stage_inequality} at timestep $t=T_1$. Then, $\left(A^\top_{T_1} \lambda + E^\top_{T_1} \nu + c_{T_1} \right)^\top x_{T_1} - \nu^\top h_{T_1}$ is a lower bound on $\tilde{G}_{T_1}(x_{T_1})$ for all $x_{T_1}$.
\end{lemma}

Lemma \ref{lemma:ddp_lb} allows one to build a collection of lower-bounding hyperplanes for $\tilde{G}_{T_1}(x_{T_1})$. Let $a = A^\top_{T_1} \lambda + E^\top_{T_1} \nu + c_{T_1}$ and $b = - \nu^\top h_{T_1}$ specify the parameters of the lower bound in Lemma \ref{lemma:ddp_lb}, found for a particular $x_{T_1}$. Each iteration of Algorithm \ref{alg:mpc_ddp_algorithm} solves \eqref{eq:second_stage_problem} for a different $x_{T_1}$, leading to a new hyperplane $(a,b)$. We can combine the hyperplanes into the set $\mathcal{H} = \{(a,b)\}$, and construct the approximate value function 

\begin{equation} \label{eq:hyperplane_underestimate}
\underline{G}_{T_1}(x_{T_1}) = \max_{(a,b) \in \mathcal{H}}(a^\top x_{T_1} + b),
\end{equation}
which is itself a lower bound on $\tilde{G}_{T_1}(x_{T_1})$. Thus, each iteration of Algorithm \ref{alg:mpc_ddp_algorithm} adds a hyperplane constraint to $\underline{G}_{T_1}(x_{T_1})$.

In words, Algorithm \ref{alg:mpc_ddp_algorithm} starts by solving the first-stage NLP with no information about the second stage. This provides a guess $x_{T_1}$ for the initial state of the second-stage LP. In step 2, the LP is solved using this $x_{T_1}$. The resulting objective, computed in step 3, is the exact second-stage value function at $x_{T_1}$. In step 4, if this objective equals the previously-found value function approximation evaluated at $x_{T_1}$, i.e. the current value function approximation is tight at the new $x_{T_1}$, then the problem terminates. Otherwise, in the key step 5 of the algorithm, Lemma \ref{lemma:ddp_lb} is used to generate a new lower-bounding hyperplane for the second-stage value function $\tilde{G}_{T_1}(x_{T_1})$. This hyperplane is then incorporated into the second-stage value function approximation $\underline{G}_{T_1}(x_{T_1})$, which is a lower bound on $\tilde{G}_{T_1}(x_{T_1})$. In step 6, we again solve the first-stage NLP, this time with the updated $\underline{G}_{T_1}(x_{T_1})$. This iterative solution of the first and second stages is repeated until the previously-found value function approximation is tight at the $x_{T_1}$ found when solving the NLP in step 6.

\begin{theorem}
Provided Assumption \ref{assumption:second_stage_feasibility} holds, steps 1 and 6 return feasible solutions of \eqref{eq:first_stage_problem_with_terminal_cost}, and step 2 solves \eqref{eq:second_stage_problem} to optimality, then Algorithm \ref{alg:mpc_ddp_algorithm} terminates in a finite number of iterations, returning a feasible solution of the split-horizon problem. In addition,
\begin{itemize}
\item[(a)] If the nonconvex first stage is solved in step 6 to local optimality, then the returned solution is locally optimal for the two-stage problem \eqref{eq:original_two_stage_problem}.
\item[(b)] If the nonconvex first stage is solved in step 6 to global optimality, then the returned solution is globally optimal for the two-stage problem \eqref{eq:original_two_stage_problem}.
\end{itemize}
\label{thm:local_solution_ddp_convergence}
\end{theorem}

\begin{proof}
We first show that Algorithm \ref{alg:mpc_ddp_algorithm} terminates in a finite number of iterations.

Suppose step 6 (or step 1 for the first iteration) has previously returned a particular feasible solution $x_{T_1}$ of \eqref{eq:first_stage_problem_with_terminal_cost}. Due to Assumption \ref{assumption:second_stage_feasibility}, solving \eqref{eq:second_stage_problem} at $x_{T_1}$ returns feasible $(\tilde{u}_{T_1},\ldots,\tilde{u}_{T-1})$ and $(\tilde{x}_{T_1+1}, \ldots, \tilde{x}_T)$. Since the second-stage LP \eqref{eq:second_stage_problem} is solved to optimality,
\begin{equation} \label{eq:second_stage_value_function_def}
\tilde{G}_{T_1}(x_{T_1}) = \sum_{t=T_1}^{T-1} \left( c_t^{\top} \tilde{x}_t + d_t^{\top} \tilde{u}_t \right) + c_T^\top \tilde{x}_T
\end{equation}
is by definition the exact second-stage value function evaluated at the particular $x_{T_1}$.

Step 5 adds the dual-feasible vertex $(\lambda,\nu)$ to $\mathcal{H}$. By the strong duality of \eqref{eq:second_stage_problem}, after $\underline{G}_{T_1}(x_{T_1})$ is updated in step 6,
\begin{equation} \label{eq:exactness_of_value_function}
\underline{G}_{T_1}(x_{T_1}) = \tilde{G}_{T_1}(x_{T_1}).
\end{equation}
at the particular $x_{T_1}$. 

Since \eqref{eq:second_stage_problem} is a linear program, there are a finite number of distinct dual-feasible vertices $(\lambda,\nu)$. If solving the LP \eqref{eq:second_stage_problem} in step 2 yields a dual-feasible vertex that has already been found in a previous iteration of the algorithm, then $\underline{G}_{T_1}(x_{T_1})$ is unchanged. If solving the NLP in step 6 returns the same $x_{T_1}$ as in the previous iteration, \eqref{eq:exactness_of_value_function} already holds. Thus, the value function approximation is tight at $x_{T_1}$, $UB = LB$, and Algorithm $\ref{alg:mpc_ddp_algorithm}$ terminates in step 4. 

Otherwise, we add a hyperplane to $\mathcal{H}$ parameterized by the new dual-feasible vertex. Since the number of such vertices is finite, one either achieves a complete characterization of $\tilde{G}_{T_1}(x_{T_1})$, or else sets $UB = LB$ before this point.

The resulting first- and second-stage arguments are each feasible, and can be concatenated (as in the algorithm) to form feasible solutions $u_\text{full}$ and $x_\text{full}$ of the split-horizon problem.

\textit{Proof of (a)}. We use Algorithm \ref{alg:mpc_ddp_algorithm} to solve \eqref{eq:original_two_stage_problem} to convergence. This returns $(u_0,\ldots,u_{T_1-1},\tilde{u}_{T_1},\ldots,\tilde{u}_{T-1})$, $(x_1,\ldots,x_{T_1},\tilde{x}_{T_1+1},\ldots,\tilde{x}_T)$, and the approximate value function $\underline{G}_{T_1}(x_{T_1})$. 

Each algorithm iteration, we solve an approximation of the first-stage NLP \eqref{eq:first_stage_problem_with_terminal_cost}, replacing $\tilde{G}_{T_1}(x_{T_1})$ with $\underline{G}_{T_1}(x_{T_1})$, and solving to local optimality. This means that when the algorithm has converged, there exist $\delta_x > 0$ and $\delta_u > 0$ such that, for all $\hat{x}_t: \|\hat{x}_t-x_t\| \leq \delta_x $ for $t = 1,\ldots,T_1$, $\hat{u}_t: \|\hat{u}_t-u_t\| \leq \delta_u$ for $t = 0,\ldots,T_1-1$,  
\begin{equation} \label{eq:first_stage_local_optimality}
\sum_{t=0}^{T_1 - 1} g_t(x_t \!,u_t) + \underline{G}_{T_1}(x_{T_1}) \leq \! \sum_{t=0}^{T_1 - 1} g_t(\hat{x}_t,\hat{u}_t) + \underline{G}_{T_1}(\hat{x}_{T_1}). 
\end{equation}

By Lemma 1, since the second stage is an LP, 
\begin{equation} \label{eq:value_function_underestimate}
\underline{G}_{T_1}(x_{T_1}) \leq \tilde{G}_{T_1}(x_{T_1}), \ \forall \ x_{T_1}.
\end{equation}

Considering the case where the algorithm has converged to a particular $x_{T_1}$, and combining \eqref{eq:exactness_of_value_function}, \eqref{eq:first_stage_local_optimality},  and \eqref{eq:value_function_underestimate}, 
\begin{align*}
&\sum_{t=0}^{T_1 - 1} g_t(x_t,u_t) + \tilde{G}_{T_1}(x_{T_1}) 
\\
& \quad \quad \quad \quad = \sum_{t=0}^{T_1 - 1} g_t(x_t,u_t) + \underline{G}_{T_1}(x_{T_1})
\\
& \quad \quad \quad \quad \leq \sum_{t=0}^{T_1 - 1} g_t(\hat{x}_t,\hat{u}_t) + \underline{G}_{T_1}(\hat{x}_{T_1})
\\
& \quad \quad \quad \quad \leq \sum_{t=0}^{T_1 - 1} g_t(\hat{x}_t,\hat{u}_t) + \tilde{G}_{T_1}(\hat{x}_{T_1}).
\end{align*}

This result holds for all $\hat{x}_t: \|\hat{x}_t-x_t\| \leq \delta_x $ for $t = 1,\ldots,T_1$, $\hat{u}_t: \|\hat{u}_t-u_t\| \leq \delta_u$ for $t = 0,\ldots,T_1-1$. Thus, $(x_1,\ldots,x_{T_1})$ and $(u_0,\ldots,u_{T_1-1})$ are locally optimal solutions to the exact \eqref{eq:first_stage_problem_with_terminal_cost}.

To complete the proof of (a), note that the linear second-stage problem \eqref{eq:second_stage_problem} is solved to optimality for the given first-stage arguments.

\textit{Proof of (b)} follows from Theorem 1 of \cite{Flamm2018}.
\end{proof}

We point out that the result holds regardless of whether the first stage is solved with a deterministic solver. Note that when solving the first stage NLP \eqref{eq:first_stage_problem_with_terminal_cost} with a given $\underline{G}_{T_1}(x_{T_1})$, a nondeterministic solver might not return the same feasible solution in subsequent iterations. Nevertheless, the solver will return an $x_{T_1}$, where either the existing hyperplanes provide a tight bound to the value function (in which case the $UB$ and $LB$ agree in Algorithm \ref{alg:mpc_ddp_algorithm}), or a new hyperplane will be added to $\underline{G}_{T_1}(x_{T_1})$. As there are a limited number of hyperplanes to add, the algorithm will converge in a finite number of iterations.

\section{Error bound on two-stage approximation}
The results in the previous section establish the properties of solutions for the two-stage problem \eqref{eq:original_two_stage_problem} returned by Algorithm \ref{alg:mpc_ddp_algorithm}. Since the original intention was to approximate the nonlinear program \eqref{eq:original_minlp}, one would also like to know how far the optimal solutions of the two-stage approximate problem \eqref{eq:original_two_stage_problem} are from those of the true nonlinear problem \eqref{eq:original_minlp}. Related results are provided in \cite{Guigues2018}, which gives approximation bounds for a setting where each stage of a multi-stage linear program is solved to a certain tolerance. Here, we extend this approach to address the case of \eqref{eq:original_two_stage_problem}, where instead of solving a linear program to a known tolerance, we approximate the second stage of the nonlinear problem \eqref{eq:original_minlp} with a linear program, and solve that linear program exactly.

For a given timestep $t$, let $\tilde{\mathcal{Z}}_t$ denote the set of arguments of the second-stage LP \eqref{eq:second_stage_problem} that satisfy \eqref{eq:second_stage_inequality}. Recall that $\mathcal{Z}_t$ denotes the set of arguments satisfying \eqref{eq:original_minlp_set_constraint} at time $t$. 

\begin{assume} \label{assumption:lp_vs_nlp_constraint_approximation}
For all $t = T_1,\ldots,T-1$, there exists a map $M_t: \tilde{\mathcal{Z}}_t \rightarrow \mathcal{Z}_t$ such that 
\begin{itemize}
\item[(a)] There exists a $\delta_t \geq 0$, such that for all $(\tilde{x}_t,\tilde{u}_t) \in \tilde{\mathcal{Z}}_t$,
\begin{equation*}
c_{t}^{\top} \tilde{x}_t + d_t^{\top} \tilde{u}_t \leq g_t(M_t(\tilde{x}_t,\tilde{u}_t)) \leq c_{t}^{\top} \tilde{x}_t + d_t^{\top} \tilde{u}_t + \delta_t.
\end{equation*}
\item[(b)] $M_t$ is surjective, i.e. for all $(x_t,u_t) \in \mathcal{Z}_t$, there exists $(\tilde{u}_t, \tilde{x}_t) \in \tilde{\mathcal{Z}}_t$ such that $M_t(\tilde{x}_t,\tilde{u}_t) = (x_t,u_t)$.
\end{itemize} 
\end{assume}
In other words, at each timestep, there is a transformation $M_t$ from feasible arguments of the LP to feasible arguments of the NLP, such that the LP objective is an underestimate of the NLP objective, with a maximum underestimate of $\delta_t$. Additionally, all feasible solutions of the NLP can be found from a feasible solution of the LP through the transformation.

Note that by \eqref{eq:intermediate_state_coupling}, at $t=T_1$, $M_{T_1}(\tilde{x}_{T_1}, \tilde{u}_{T_1})$ leaves $\tilde{x}_{T_1}$ unchanged, and maps the input $\tilde{u}_{T_1}$ to a feasible input of the NLP.


We first consolidate the terminology for the exact NLP \eqref{eq:original_minlp} and two-stage approximation \eqref{eq:original_two_stage_problem}:
\begin{itemize} 
\item $G_0(x_0)$ is the optimal value of the exact NLP \eqref{eq:original_minlp}, as a function of the initial state $x_0$.
\item $G_{T_1}(x_{T_1})$ is the value of the second part of the exact NLP \eqref{eq:original_minlp} ($t = T_1,\ldots, T-1$), as a function of the intermediate state $x_{T_1}$. Note that 
\begin{align} \label{eq:original_minlp_two_stage}
G_{0}(x_{0}) = & \min_{\substack{u_0,\ldots,u_{T_1-1} \\ x_1,\ldots,x_{T_1}}} \
  \sum_{t=0}^{T_1 - 1} g_t(x_t,u_t) + G_{T_1}(x_{T_1})
\\ \nonumber
\text{s.t.} \quad &  x_{t+1} = f_t(x_t,u_t), \ \ t = 0, \ldots, T_1-1
\\ \nonumber
&  (x_t,u_t) \in \mathcal{Z}_t, \quad \quad  \ \, t = 0, \ldots, T_1-1.
\end{align}
\item $\tilde{G}_{T_1}(x_{T_1})$ is the value of the second-stage LP \eqref{eq:second_stage_problem}, as a function of the intermediate state $x_{T_1}$.
\item $\underline{G}_{T_1}(x_{T_1})$ is the approximate value function for the second-stage LP \eqref{eq:second_stage_problem} in Algorithm \ref{alg:mpc_ddp_algorithm}. It consists of a finite number of hyperplanes bounding $\tilde{G}_{T_1}(x_{T_1})$ from below. Hyperplanes are iteratively added to this set via DDP.
\end{itemize}
Now, we consider the objective for the split-horizon problem \eqref{eq:original_two_stage_problem} after Algorithm \ref{alg:mpc_ddp_algorithm} has converged, and define it as
\begin{align} \label{eq:second_stage_approx_problem}
\underline{G}_{0}(x_{0}) = & \min_{\substack{u_0,\ldots,u_{T_1-1} \\ x_1,\ldots,x_{T_1}}} \  \sum_{t=0}^{T_1 - 1} g_t(x_t,u_t) + \underline{G}_{T_1}(x_{T_1})
\\ \nonumber
\text{s.t.} \quad &  x_{t+1} = f_t(x_t,u_t), \ \ t = 0, \ldots, T_1-1
\\ \nonumber
&  (x_t,u_t) \in \mathcal{Z}_t, \quad \quad  \ \, t = 0, \ldots, T_1-1.
\end{align}

Our aim is to bound the difference between $G_{0}(x_{0})$ and $\underline{G}_{0}(x_{0})$. 

\begin{theorem} \label{thm:two_stage_error_bound}
If Assumptions \ref{assumption:second_stage_feasibility} and \ref{assumption:lp_vs_nlp_constraint_approximation} hold, and step 6 is solved to $\epsilon$-optimality in each iteration, then
\begin{equation*}
\underline{G}_{0}(x_{0}) \leq G_{0}(x_{0}) \leq \underline{G}_{0}(x_{0}) + \epsilon + \sum_{t=T_1}^{T-1} \delta_t.
\end{equation*} 
\end{theorem}

\begin{proof}
By the definitions of \eqref{eq:original_minlp} and \eqref{eq:original_two_stage_problem}, the NLP and split-horizon problems are equivalent for the first part of the horizon, and differ only in the second stage. 

By Lemma 1, the Benders decomposition underestimates the objective of the second-stage LP, and $\underline{G}_{T_1}(x_{T_1}) \leq \tilde{G}_{T_1}(x_{T_1})$. By Assumption \ref{assumption:lp_vs_nlp_constraint_approximation}, for any solution to the true NLP, a corresponding solution to the LP approximation exists with a lower objective. Combining these two observations, for any $x_{T_1}$, 
\begin{equation} \label{eq:second_stage_underapproximation}
\underline{G}_{T_1}(x_{T_1}) \leq \tilde{G}_{T_1}(x_{T_1}) \leq G_{T_1}(x_{T_1}).
\end{equation}

If we solve \eqref{eq:second_stage_approx_problem} to $\epsilon$-optimality in step 6, returning $(u_0,\ldots,u_{T_1-1})$ and $(x_1,\ldots,x_{T_1})$, then
\begin{equation} \label{eq:value_function_label_4}
 \underline{G}_{0}(x_{0}) \leq \sum_{t=0}^{T_1 - 1} g_t(x_t,u_t) + \underline{G}_{T_1}(x_{T_1}) \leq \underline{G}_{0}(x_{0}) + \epsilon.
\end{equation}

By \eqref{eq:second_stage_underapproximation}, since $(u_0,\ldots,u_{T_1-1})$ and $(x_1,\ldots,x_{T_1})$ meet the constraints of \eqref{eq:original_minlp_two_stage} (and \eqref{eq:second_stage_approx_problem}), the objective of \eqref{eq:second_stage_approx_problem} is less than or equal to the objective of \eqref{eq:original_minlp_two_stage}, and thus
\begin{align} \label{eq:value_function_label_1}
0 & \leq G_{0}(x_{0}) - \underline{G}_{0}(x_{0}) 
\\ \label{eq:value_function_label_5}
& \leq G_{0}(x_{0}) - \sum_{t=0}^{T_1 - 1} g_t(x_t,u_t) - \underline{G}_{T_1}(x_{T_1}) + \epsilon
\\ \label{eq:value_function_label_2}
&\leq G_{T_1}(x_{T_1}) - \underline{G}_{T_1}(x_{T_1}) + \epsilon.
\end{align}

The second inequality \eqref{eq:value_function_label_5} comes from substituting in the right inequality of \eqref{eq:value_function_label_4} into \eqref{eq:value_function_label_1}. The final inequality \eqref{eq:value_function_label_2} is because $G_0(x_0) \leq \sum_{t=0}^{T_1 - 1} g_t(x_t,u_t) + G_{T_1}(x_{T_1})$ for any feasible solution, in particular the found $\epsilon$-suboptimal solution.

Assumption \ref{assumption:lp_vs_nlp_constraint_approximation} states that at each second-stage timestep, all feasible arguments of $G_{T_1}(\cdot)$ are the result of the mapping $M_t$ from feasible arguments of $\tilde{G}_{T_1}(\cdot)$, such that the objective of $\tilde{G}_{T_1}(\cdot)$ is an underestimate of the objective of $G_{T_1}(\cdot)$ by at most $\delta_t$. Taking the worst case for each timestep from $T_1$ to $T-1$, for any $x_{T_1}$,
\begin{equation}
G_{T_1}(x_{T_1}) - \tilde{G}_{T_1}(x_{T_1}) \leq \sum_{t=T_1}^{T-1} \delta_t.
\end{equation}

When Algorithm \ref{alg:mpc_ddp_algorithm} has converged, by \eqref{eq:exactness_of_value_function}, $\tilde{G}_{T_1}(x_{T_1}) = \underline{G}_{T_1}(x_{T_1})$, and thus 
\begin{equation} \label{eq:value_function_label_3}
G_{T_1}(x_{T_1}) - \underline{G}_{T_1}(x_{T_1}) \leq \sum_{t=T_1}^{T-1} \delta_t.
\end{equation}
Combining \eqref{eq:value_function_label_1}, \eqref{eq:value_function_label_2}, and \eqref{eq:value_function_label_3},
\begin{equation}
\underline{G}_{0}(x_{0}) \leq G_{0}(x_{0}) \leq \underline{G}_{0}(x_{0}) + \epsilon + \sum_{t=T_1}^{T-1} \delta_t.
\end{equation}
\end{proof}

We now consider the difference between the optimum of the exact problem, and the objective of the exact problem evaluated using the arguments found in the approximate problem.

\begin{theorem} \label{theorem:two_stage_argument_optimality_bound}
If Assumptions \ref{assumption:second_stage_feasibility} and \ref{assumption:lp_vs_nlp_constraint_approximation} hold, step 6 is solved to $\epsilon$-optimality, and step 6 returns $(u_0,\ldots,u_{T_1-1})$ and $(x_1,\ldots,x_{T_1})$ when Algorithm \ref{alg:mpc_ddp_algorithm} converges, then 
\begin{equation*}
G_0(x_0) \leq \! \! \sum_{t=0}^{T_1 - 1} \! \!g_t(x_t, u_t) + G_{T_1}(x_{T_1}) \leq G_0(x_0) + \epsilon + \! \sum_{t=T_1}^{T-1} \! \delta_t.
\end{equation*} 
\end{theorem}

\begin{proof}
The first inequality is straightforward: $G_0(x_0)$ is the optimal objective of the NLP \eqref{eq:original_minlp}. Evaluating the objective of the NLP at any other feasible solution cannot lead to a lower objective.

Now, for the sake of notational ease, denote the resulting arguments when step 6 is solved to $\epsilon$-optimality as $u = (u_0,\ldots,u_{T_1-1})$ and $x = (x_1,\ldots,x_{T_1})$. Furthermore, denote the first-stage cost as $C_0(x,u) = \sum_{t=0}^{T_1 - 1} g_t(x_t,u_t)$. 

Due to the $\epsilon$-optimality of $(x,u)$, for all feasible $(\hat{x},\hat{u})$,
\begin{equation} \label{eq:lemma_3_epsilon_suboptimality}
C_0(x,u) + \underline{G}_{T_1}(x_{T_1}) \leq C_0(\hat{x},\hat{u}) + \underline{G}_{T_1}(\hat{x}_{T_1}) + \epsilon.
\end{equation}

To prove the second inequality in the theorem, consider the difference between the NLP objective evaluated using $(x,u)$, and the optimal NLP objective $G_0(x_0)$, with corresponding arguments $(x^*,u^*)$:
{\allowdisplaybreaks
\begin{align} \nonumber
&C_0(x,u) + G_{T_1}(x_{T_1}) - G_0(x_0) 
\\ \label{eq:lemma_3_first_inequality}
&\leq C_0(x,u) + \underline{G}_{T_1}(x_{T_1}) + \sum_{t=T_1}^{T-1} \delta_t - G_0(x_0) 
\\ \label{eq:lemma_3_second_inequality}
&\leq C_0(x^{*},u^{*}) + \underline{G}_{T_1}(x_{T_1}^{*}) + \epsilon + \sum_{t=T_1}^{T-1} \delta_t
- G_0(x_0)
\\ \nonumber
&= C_0(x^{*},u^{*}) + \underline{G}_{T_1}(x_{T_1}^{*}) + \epsilon + \sum_{t=T_1}^{T-1} \delta_t
\\ \label{eq:lemma_3_third_inequality}
& \quad - C_0(x^{*},u^{*}) - G_{T_1}(x_{T_1}^{*}) 
\\ \nonumber
&= \epsilon + \sum_{t=T_1}^{T-1} \delta_t + \underline{G}_{T_1}(x_{T_1}^{*}) - G_{T_1}(x_{T_1}^{*})  
\\ \label{eq:lemma_3_fourth_equality}  
&\leq \epsilon + \sum_{t=T_1}^{T-1} \delta_t.
\end{align}
}

The inequality \eqref{eq:lemma_3_first_inequality} comes from substituting \eqref{eq:value_function_label_3} for $G_{T_1}(x_{T_1})$. The inequality \eqref{eq:lemma_3_second_inequality} is due to choosing $(x^*,u^*)$ in \eqref{eq:lemma_3_epsilon_suboptimality}. The equality \eqref{eq:lemma_3_third_inequality} comes from expanding $G_0(x_0)$ in terms of its optimal arguments. Finally, \eqref{eq:lemma_3_fourth_equality} holds because $\underline{G}_{T_1}(x_{T_1}^{*}) \leq G_{T_1}(x_{T_1}^{*})$, as in \eqref{eq:second_stage_underapproximation}.
\end{proof}

By solving the first stage to a higher level of optimality and using a tighter linearization, the difference between the exact and approximate problem can be reduced. Nevertheless, the numerical results presented below suggest that very good performance can be obtained for a relatively coarse approximation (see Section \ref{section:multicell_comparison} for a comparison with a state-of-the-art method).

\section{Hydro Reservoir System Model and Approximation}
We apply the above split-horizon approximation method to a system of $N$ interconnected reservoirs, with a specific topology of pumps and turbines that allow the reservoirs to exchange water. Denote the set of reservoirs to which reservoir $i$ can charge or discharge  as $\mathcal{N}^{i \to}$ and the set of reservoirs which can charge or discharge to reservoir $i$ as $\mathcal{N}^{\to i}$. We can transfer water volume $V^{i \rightarrow j}$ from each reservoir $i$ to any $j \in  \mathcal{N}^{i \to}$. The power associated with this action, $P^{i \rightarrow j}$, is positive when pumping water to a higher elevation reservoir or negative when releasing water through a turbine to a lower elevation reservoir. We buy and sell power on the electricity spot market at price $p$. This convention on device power and electricity price means that a negative objective value corresponds to a profit, while a positive objective value represents a loss. The aim is to maximize profit by taking advantage of spot price fluctuations. We assume the volume stored in reservoir $i$ is linearly proportional to the reservoir level $\ell^i$, with a proportionality constant $\gamma^i$ reflecting the surface area of the reservoir. That is, if $V^{i \rightarrow j}$ is transferred from reservoir $i$ to $j$, $\ell^i$ will decrease by $V^{i \rightarrow j}/ \gamma^i $ (and $\ell^j$ will increase by $V^{i \rightarrow j}/ \gamma^j $). For simplicity, we assume that there are no inflows or evaporative losses in this system.

Intuitively, the problem can be cast in the finite horizon optimal control framework considered here by setting $x = (\ell_0,\ldots,\ell_T)$, with $\ell_t = (\ell_t^1,\ldots,\ell_t^N)$ for all $t$ and $u = (V_0,\ldots,V_{T-1})$, with $V_t = (\{V^{i \to j}_t\}_{j \in \mathcal{N}^{i \to}})_{i = 1,\ldots,N}$. We wish to consider a horizon that is sufficiently long to capture seasonal price fluctuations.

\subsection{Nonlinear Exact Model}
We assume that the energy associated with transferring a unit volume of water from reservoir $i$ to $j$ depends on the net head between reservoirs $i$ and $j$ \cite{Cerisola2012}:
\begin{equation*}
E^{i \rightarrow j}_t = \alpha^{i \rightarrow j} + \beta^{i \rightarrow j} \left( \ell^i_t \!-\! \ell^j_t \right).
\end{equation*}
Letting $P^{i \rightarrow j}_t = V^{i \rightarrow j}_t E^{i \rightarrow j}_t$, we write the optimal control problem as
\begin{subequations} \label{eq:bilinear_problem}
\begin{align} \label{eq:bilinear_objective}
\min_{P,\ell,V} \ & \sum_{t=0}^{T-1} \sum_{i=1}^N  \sum_{ j \in \mathcal{N}^{i \to}} p_t P^{i \rightarrow j}_t + c_T^{\top} \ell_T
\\ \nonumber
\text{s.t.} \quad & P^{i \rightarrow j}_t = V^{i \rightarrow j}_t \left( \alpha^{i \to j} + \beta^{i \to j} \left( \ell^i_t-\ell^j_t \right) \right) 
\\ \label{eq:bilinear_constraint_power}
& \quad \forall \ j \in \mathcal{N}^{i \to}
\\ \label{eq:level_dynamics}
& \ell^i_{t+1} = \ell^i_t + \frac{1}{\gamma_i} \Big( \sum_{k \in \mathcal{N}^{\to i}} \! V^{k \rightarrow i}_t - \! \sum_{j \in \mathcal{N}^{i \to}} \! V^{i \rightarrow j}_t \Big)
\\ \label{eq:level_bounds}
& \underline{\ell}^i \leq \ell^i_t \leq \bar{\ell}^i, \quad \ell^i_0 \text{ given}
\\ \label{eq:vol_bounds}
& \underline{V}^{i \to j} \leq V^{i \rightarrow j}_t \leq \bar{V}^{i \to j} \, \quad \forall \ j \in \mathcal{N}^{i \to}
\\ \nonumber
& \quad i = 1, \ldots, N;  \quad t = 0, \ldots, T-1.
\end{align}
\end{subequations}
By absorbing equation \eqref{eq:bilinear_constraint_power} into the objective function, we can move the bilinear terms from the constraints into the objective, and eliminate the variables $P_t^{i \to j}$. The objective \eqref{eq:bilinear_objective} then becomes
\begin{align} \label{eq:bilinear_objective_split}
\sum_{t=0}^{T-1} \sum_{i=1}^N  \sum_{ j \in \mathcal{N}^{i \to}} p_t V^{i \rightarrow j}_t \alpha^{i \to j} + c_T^{\top} \ell_T + \textit{bilinear term}
\end{align}
with the bilinear term equal to
\begin{align} \nonumber
&\sum_{t=0}^{T-1} \sum_{i=1}^N  \sum_{ j \in \mathcal{N}^{i \to}} p_t V^{i \rightarrow j}_t  \beta^{i \to j} \left( \ell^i_t-\ell^j_t \right)
\\ \label{eq:bilinear_objective_bilinear_part}
= &\sum_{t=0}^{T-1} \sum_{i=1}^N  \sum_{ j \in \mathcal{N}^{i \to}} p_t \beta^{i \to j} \left( V^{i \rightarrow j}_t \ell^i_t- V^{i \rightarrow j}_t \ell^j_t \right).
\end{align}
The reformulated problem thus has linear constraints (since \eqref{eq:bilinear_constraint_power} has been removed), and an objective with linear and bilinear terms. For sufficiently small horizons, we can find the global optimum in a reasonable length of time using \textsc{YALMIP}'s built-in \textit{BMIBNB} spatial branch-and-bound solver \cite{Lofberg2004}. We have observed that Matlab's \textit{fmincon} also tends to return the global optimum for small problem instances in the simulations presented in Section \ref{section:results}.

\subsection{Linear Approximate Model} \label{section:linearization_subsection}
A common way of solving problem \eqref{eq:bilinear_problem} is to solve a convex outer approximation. One way to do this is to replace the constraint \eqref{eq:bilinear_constraint_power} with
\begin{align} \label{eq:crossterm_linearization}
P^{i \rightarrow j}_t =  V^{i \rightarrow j}_t \; \alpha^{i \to j} + \beta^{i \to j} (\chi^{(i \to j,i)}_t-\chi^{(i \to j,j)}_t).
\end{align}
Here, we have introduced the additional variables $\chi^{(i \rightarrow j,i)}_t$ and $\chi^{(i \rightarrow j,j)}_t$  to represent the bilinear terms $V^{i \rightarrow j}_t \ell^i_t$ and $V^{i \rightarrow j}_t \ell^j_t$ respectively. Incorporating this linear constraint into the objective, the bilinear term \eqref{eq:bilinear_objective_bilinear_part}, is thus approximated by 
\begin{align} \label{eq:bilinear_objective_linearized_part}
\sum_{t=0}^{T-1} \sum_{i=1}^N  \sum_{ j \in \mathcal{N}^{i \to}} p_t \beta^{i \to j} \left( \chi^{(i \rightarrow j,i)}_t - \chi^{(i \rightarrow j,j)}_t \right).
\end{align}
We then linearize using McCormick envelopes \cite{Cerisola2012}, adding to the problem for each $\chi^{(i \rightarrow j,i)}_t$ the inequalities 
\begin{subequations} \label{eq:mccormick_envelope}
\begin{align} \label{eq:mccormick_1}
\chi^{(i \rightarrow j,i)}_t &\geq V^{i \rightarrow j}_t \bar{\ell}^i_t + \bar{V}^{i \rightarrow j}_t  \ell^i_t  - \bar{V}^{i \rightarrow j}_t  \bar{\ell}^i_t, 
\\ \label{eq:mccormick_2}
\chi^{(i \rightarrow j,i)}_t &\geq V^{i \rightarrow j}_t \underline{\ell}^i_t  + \underline{V}^{i \rightarrow j}_t  \ell^i_t  - \underline{V}^{i \rightarrow j}_t  \underline{\ell}^i_t, 
\\ \label{eq:mccormick_3}
\chi^{(i \rightarrow j,i)}_t &\leq V^{i \rightarrow j}_t \bar{\ell}^i_t + \underline{V}^{i \rightarrow j}_t  \ell^i_t - \underline{V}^{i \rightarrow j}_t  \bar{\ell}^i_t, 
\\ \label{eq:mccormick_4}
\chi^{(i \rightarrow j,i)}_t &\leq V^{i \rightarrow j}_t \underline{\ell}^i_t  + \bar{V}^{i \rightarrow j}_t  \ell^i_t - \bar{V}^{i \rightarrow j}_t  \underline{\ell}^i_t. 
\end{align}
\end{subequations}
Here, the upper and lower McCormick bounds, $[\underline{V}^{i \rightarrow j}_t, \bar{V}^{i \rightarrow j}_t] $ for $V^{i \rightarrow j}_t$, and $[\underline{\ell}^i_t, \bar{\ell}^i_t]$ for $\ell^i_t$, can depend on the timestep. This results in a linear program, with objective \eqref{eq:bilinear_objective_split} (with the bilinear term replaced by \eqref{eq:bilinear_objective_linearized_part}), and constraints \eqref{eq:level_dynamics}-\eqref{eq:vol_bounds}, \eqref{eq:mccormick_envelope}, and a similar set of McCormick bounds for each $\chi^{(i \rightarrow j,j)}_t$.

\subsection{Error Bound on McCormick Envelope Approximation}
We now calculate the maximum error from approximating the bilinear relation $\chi^{(i \rightarrow j,i)}_t = V^{i \rightarrow j}_t \ell^i_t$ by its McCormick envelope in \eqref{eq:mccormick_envelope}.  We drop the time, source, and destination indices in the subsequent analysis for cleaner notation.

For each bounding hyperplane in \eqref{eq:mccormick_envelope}, the points where the first derivative of the approximation error with respect to either $V$ or $\ell$ is zero have zero approximation error. For example, for the first hyperplane \eqref{eq:mccormick_1}, the approximation error is $(V \bar{\ell} + \bar{V}  \ell  - \bar{V}  \bar{\ell}) - V \ell$. The  first derivatives are zero where $\ell = \bar{\ell}$ or $V = \bar{V}$, where the error is also zero.

Thus, the largest errors must occur at the intersection of the bounding hyperplanes. Considering Figure \ref{fig:mccormick_envelope_visualization}, the largest overestimate must occur at the intersection of the two upper-bounding hyperplanes \eqref{eq:mccormick_3} and \eqref{eq:mccormick_4}, and the largest underestimate at the intersection of the two lower-bounding hyperplanes \eqref{eq:mccormick_1} and \eqref{eq:mccormick_2} (the intersection of upper- and lower-bounding hyperplanes, e.g., \eqref{eq:mccormick_1} and \eqref{eq:mccormick_3}, is an exact representation of the underlying function).

\begin{figure}[tbp] 
   \includegraphics[width=0.235\textwidth]{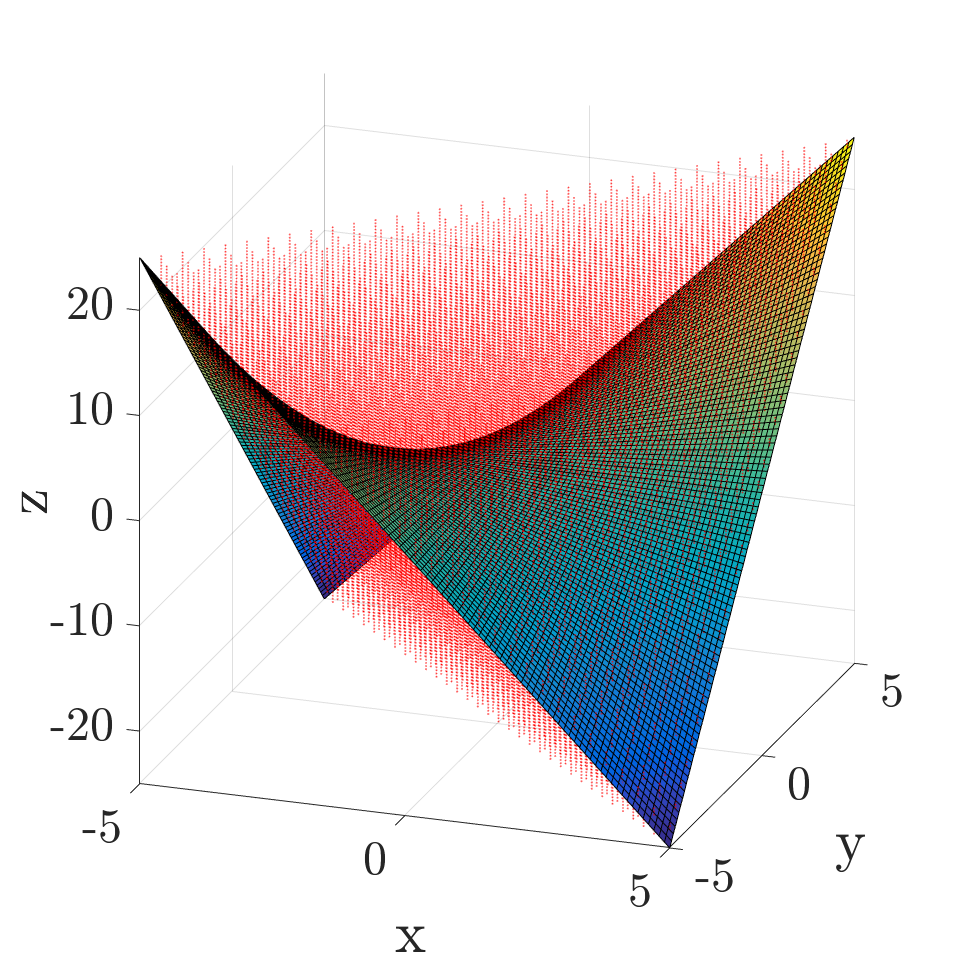} 
   \hfill
   \includegraphics[width=0.235\textwidth]{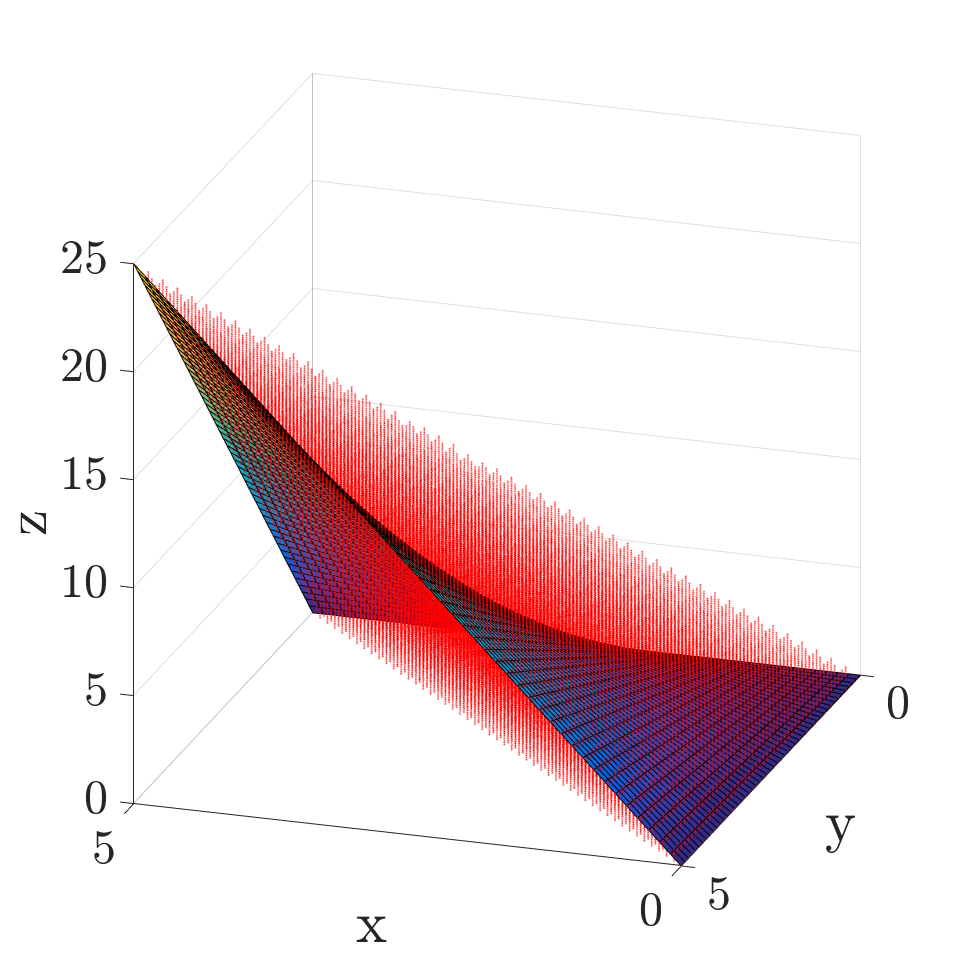}
   \caption{McCormick envelope convex approximation of $z = xy$, over various bounds.}
   \label{fig:mccormick_envelope_visualization}
\end{figure}

We determine the largest overestimate analytically. At the intersection of the two upper-bounding hyperplanes, 
$$V \bar{\ell} + \underline{V} \ell - \underline{V}  \bar{\ell} = V  \underline{\ell}  + \bar{V}  \ell - \bar{V}  \underline{\ell}.$$

Solving for $V$ in terms of $\ell$,
$$V_{cd}(\ell) = \frac{\ell (\bar{V} - \underline{V}) + \underline{V} \bar{\ell} - \bar{V} \underline{\ell}}{\bar{\ell} - \underline{\ell}}$$

The approximation error at the intersection of the two upper-bounding hyperplanes is
\begin{align*}
\bar{e}(\ell) &= (V_{cd}(\ell) \bar{\ell} + \underline{V} \ell - \underline{V}  \bar{\ell})  - V_{cd}(\ell) \ell 
\\
&= \left( \frac{\ell (\bar{V} - \underline{V}) + \underline{V} \bar{\ell} - \bar{V} \underline{\ell}}{\bar{\ell} - \underline{\ell}} \right) (\bar{\ell} - \ell) + \underline{V} \ell - \underline{V}  \bar{\ell}.
\end{align*}

Taking the derivative with respect to $\ell$, we have $\frac{d \bar{e}(\ell)}{d \ell} = -2 \ell + \bar{\ell} + \underline{\ell}$. The derivative of the approximation error is zero at $\ell^* = \left( \bar{\ell} + \underline{\ell} \right)/2$ with the corresponding 
$V^* = \left( \bar{V} + \underline{V} \right)/2$. This leads to a maximum overestimate of 
\begin{equation} \label{eq:max_approx_overestimate}
\bar{e}^* = (\bar{V} - \underline{V})(\bar{\ell} - \underline{\ell})/4.
\end{equation}
By similar calculations, or by symmetry, the maximum underestimate $\underline{e}^*$ occurs at the same $V^*$ and $\ell^*$, with a negative sign in the error.

\subsection{Error Bound on Two-Stage Approximation} 
We solve the two-stage approximate problem using Algorithm \ref{alg:mpc_ddp_algorithm}. Since the two-stage approximation differs from the exact problem only in the second stage, a bound on the two-stage approximation error can be found by summing the worst-case approximation error \eqref{eq:max_approx_overestimate} for each second-stage bilinear term in \eqref{eq:bilinear_objective_bilinear_part}. In the context of Assumption \ref{assumption:lp_vs_nlp_constraint_approximation}, at each second-stage timestep $t$, the worst-case bound between the exact and approximate objectives is
\begin{align} \label{eq:mccormick_approximation_delta_t}
\delta_t \! = \! \sum_{i=1}^N  \sum_{ j \in \mathcal{N}^{i \to}} \! \! \! \frac{|p_t \beta^{i \to j}|}{4} \! \left( \bar{V}^{i \to j}_t \! - \underline{V}^{i \to j}_t \right) \! \! \left( \bar{\ell}^i_t - \underline{\ell}^i_t  +  \bar{\ell}^j_t - \underline{\ell}^j_t \right). 
\end{align}
Solutions to the two-stage approximation are also feasible for the exact problem, so the map $M_t$ between solutions of the two problem formulations is simply identity.

 We assume that the first stage is solved exactly, i.e., \mbox{$\epsilon = 0$}. Denoting the approximate solution as $\underline{G}_0(x_0)$ and optimal solution as $G_0(x_0)$, and applying Theorem \ref{thm:two_stage_error_bound}, $\underline{G}_0(x_0)$ achieves a maximum underestimate of 
\begin{align} \label{eq:mccormick_approximation_error_bound}
\sum_{t=T_1}^{T-1} \sum_{i=1}^N  \sum_{ j \in \mathcal{N}^{i \to}} \! \! \! \frac{|p_t \beta^{i \to j}|}{4} \! \left( \bar{V}^{i \to j}_t \! - \underline{V}^{i \to j}_t \right) \! \! \left( \bar{\ell}^i_t - \underline{\ell}^i_t  +  \bar{\ell}^j_t - \underline{\ell}^j_t \right). 
\end{align}

The absolute limits on the reservoir levels \eqref{eq:level_bounds} and volume flows \eqref{eq:vol_bounds} can be used in the McCormick envelopes, but these crude estimates can be improved. Tighter bounds on the volume flows are not known a priori, but the bounds on the reservoir levels can be tightened by incorporating the system dynamics \eqref{eq:level_dynamics}, and assuming maximum inflows and outflows at each timestep:
\begin{align} \label{eq:improving_mccormick_bounds_1}
\bar{\ell}^i_t &= \min \left( \ell^i_0 + \frac{1}{\gamma^i} \sum_{s = 0}^{t-1} \sum_{j \in \mathcal{N}^{\to i}} \bar{V}^{j \rightarrow i}_s, \ \bar{\ell}^i \right)
\\ \label{eq:improving_mccormick_bounds_2}
\underline{\ell}^i_t &= \max \left( \ell^i_0 - \frac{1}{\gamma^i} \sum_{s = 0}^{t-1} \sum_{j \in \mathcal{N}^{i \to}} \bar{V}^{i \rightarrow j}_s, \ \underline{\ell}^{i} \right).
\end{align}

For the split-horizon problem, the bounds must be found relative to the starting level for the entire problem, rather than the starting level for the linear part. Otherwise, the approximation of the bilinear terms changes at each iteration, leading to a different linear program each iteration and voiding the convergence guarantee of Theorem \ref{thm:local_solution_ddp_convergence}.

\section{Numerical Results} \label{section:results}

\subsection{Experimental System}
We apply the split-horizon approximation and solution method of Algorithm~1 to the system of two reservoirs and an infinite-capacity basin depicted in Figure \ref{fig:test_system_diagram}. This example system is inspired by \cite{Borghetti2008}, which considers a single reservoir of the same volume and height used here. The reservoirs have identical capacities of $33 \times 10^6 \ \SI{}{\meter}^3$, and their bottoms are separated by a height of $h_0 = \SI{200}{\meter}$. The water volume in each reservoir is proportional to the water level relative to the reservoir bottom, with maximum levels $\bar{\ell}^a = \SI{85}{\meter}$ and $\bar{\ell}^b = \SI{100}{\meter} $ for the upper and lower reservoir, respectively. The reservoir bottoms are connected by a reversible $\SI{100}{\mega\watt}$ pump/turbine which operates with $\mu = 90\%$ one-way conversion efficiency. The lower reservoir is connected to an infinite basin ($\bar{\ell}^{bas} = \underline{\ell}^{bas} = \SI{0}{\meter}$) situated $\SI{300}{\meter}$ below its bottom, via another identical reversible $\SI{100}{\mega\watt}$ pump/turbine.  The reservoir levels are constrained between empty ($\underline{\ell}^{a} = \underline{\ell}^{b} = \SI{0}{\meter}$) and their maximums. The reservoirs are initially half-full.

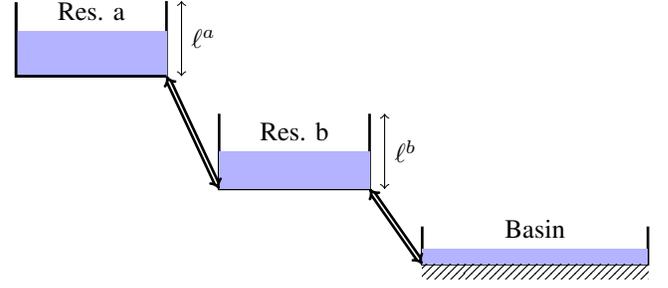
\begin{figure} 
	\begin{centering}
	\begin{tikzpicture}
		
	\newcommand\mypool[8]{
\draw[line width=1pt] (#1) -- +(0,#2)coordinate[](a){} --node[above=0.2+#7cm, pos=0.5]{#6} +(#4,#2)coordinate[](#5){} -- +(#4,0);
\draw[<->,line width=1pt,thin] ([xshift=#4cm+0.2cm]a) -- ([xshift=#4cm+0.2cm,yshift=-#2cm]a)node[right=0.1mm, pos = 0.5](){#3};
\path[fill=blue!30](a) -- ([yshift=#7cm]a)-- ([yshift=#7cm]#5) -- (#5)--cycle;
\draw[line width=1pt]([yshift=1cm]#1) |-(#5);
}	
	
	\newcommand\mypoolturbine[9]{
\draw[line width=1pt] ([xshift=#9cm,yshift=-#8cm]#1) -- +(0,#2)coordinate[](a){} --node[above=0.2+#7cm, pos=0.5]{#6} +(#4,#2)coordinate[](#5){} -- +(#4,0);
\draw[<->,line width=1pt,thin] ([xshift=#4cm+0.2cm]a) -- ([xshift=#4cm+0.2cm,yshift=-#2cm]a)node[right=0.1mm, pos = 0.5](){#3};
\path[fill=blue!30](a) -- ([yshift=#7cm]a)-- ([yshift=#7cm]#5) -- (#5)--cycle;
\draw[implies-implies,line width=1pt,double] (#1) -- ([xshift=#9cm,yshift=-#8cm+#2cm]#1);
}		

\newcommand\mypoolturbineend[9]{
\draw[line width=1pt] ([xshift=#9cm,yshift=-#8cm]#1) -- +(0,#2)coordinate[](a){} --node[above=0.2+#7cm, pos=0.5]{#6} +(#4,#2)coordinate[](#5){} -- +(#4,0);
\path[fill=blue!30](a) -- ([yshift=#7cm]a)-- ([yshift=#7cm]#5) -- (#5)--cycle;
\draw[implies-implies,line width=1pt,double] (#1) -- ([xshift=#9cm,yshift=-#8cm+#2cm]#1);
\path[pattern=north east lines](a) -- ([yshift=#3cm]a)-- ([yshift=#3cm]#5) -- (#5)--cycle;
}	
	
\node[inner sep=0pt, outer sep=0pt] at (0,0) (O){};
\mypool{O}{-1}{$\ell^a$}{2}{B}{Res. a}{0.6}{0.2}
\mypoolturbine{B}{-1}{$\ell^b$}{2}{C}{Res. b}{0.5}{0.5}{0.7}
\mypoolturbineend{C}{-0.5}{-0.2}{3}{D}{Basin}{0.2}{0.5}{0.7}
	
	\end{tikzpicture}
	
	\caption{Test system with two reservoirs and reversible turbines/pumps, and an infinite basin.}
	\label{fig:test_system_diagram}
	\end{centering}
\end{figure}

The energy conversion parameters for \eqref{eq:bilinear_constraint_power} can be calculated from first principles, since the energy required to raise a given volume a height $h$ is $E = \rho g h$, with $\rho = \SI{1000}{\kg/\meter^3}$ and $g = \SI{9.81}{\meter/\second^2}$. Taking into account the one-way conversion efficiency $\mu$, the energy conversion parameters for \eqref{eq:bilinear_constraint_power} are $\alpha^{a \to b} = -k_{H_2O} \, h_0 \, \mu$, $\beta^{a \to b} = -k_{H_2O} \, \mu$, $\alpha^{b \to a} = k_{H_2O} \, h_0 / \, \mu$, and $\beta^{b \to a} = k_{H_2O}/ \, \mu$, where $k_{H_2O} = \SI{0.002725}{\kilo\watt}$h. The maximum flow rates in \eqref{eq:vol_bounds} are calculated as the flows that produce or consume $\SI{100}{\mega\watt}$ at the initial reservoir levels. Reservoir $a$ takes 156 hours to empty into Reservoir $b$ at maximum release flow, while Reservoir $b$ takes 283 hours to empty into the lower basin.

The terminal cost $c_T$ of water stored in reservoirs is chosen as the energy embodied in a unit volume of water at the beginning of the time horizon, multiplied by the average price throughout the time horizon. This energy includes conversion efficiency losses, and is relative to the zero potential level of the system. Thus, the value of water stored in the upper reservoir takes into account its potential passage through the lower reservoir as well. Price data are taken from the Swiss day-ahead EPEX spot market, starting on January 1, 2017 \cite{SwissSpotPrice}, and are assumed to be perfectly known over the entire horizon. 

\subsection{Solution of Nonlinear and Split-Horizon Problems}

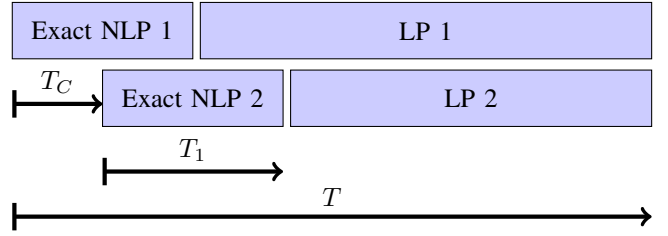
\begin{figure}[tbp]
	\begin{centering}
	\begin{tikzpicture}
        
        \draw[|->,ultra thick] (0,-0.3)--(8.5,-0.3) node[above,pos=.5]{$T$};     
        
        \draw[|->,ultra thick] (1.2,0.3)--(3.6,0.3) node[above,pos=.5]{$T_1$};   
        
        \draw[|->,ultra thick] (0,1.2)--(1.2,1.2) node[above,pos=.5]{$T_C$};   		
  		
		\filldraw[fill=blue!20, draw=black] (1.2,0.9) rectangle (3.6,1.65) node[pos=.5] {Exact NLP 2};
  		\filldraw[fill=blue!20, draw=black] (3.7,0.9) rectangle (8.5,1.65) node[pos=.5] {LP 2};
  		
  		\filldraw[fill=blue!20, draw=black] (0,1.8) rectangle (2.4,2.55) node[pos=.5] {Exact NLP 1};
  		\filldraw[fill=blue!20, draw=black] (2.5,1.8) rectangle (8.5,2.55) node[pos=.5] {LP 1};

	\end{tikzpicture}
	
	\caption{Depiction of shrinking-horizon simulation over two simulation steps. Simulation is run over total horizon $T$, with nonlinear first stage of length $T_1$, and control horizon of length $T_C$ where current solution is applied.}
	\label{fig:ddp_shrinking_horizon_method_illustration}
	\end{centering}
\end{figure}

We seek to minimize the exact NLP \eqref{eq:bilinear_problem}, with the total horizon $T = 20$ days. As a reference estimate of the global optimal value, we grid both the state and input space into 32 equally-spaced points and solve the approximate dynamic program using the DPM toolbox \cite{Sundstrom2009}. We compare this to the solution of the split-horizon problem (consisting of \eqref{eq:bilinear_problem} with modifications \eqref{eq:crossterm_linearization}, \eqref{eq:bilinear_objective_linearized_part}, and \eqref{eq:mccormick_envelope}) over a shrinking horizon, as depicted in Figure \ref{fig:ddp_shrinking_horizon_method_illustration}. Here, we solve the split-horizon problem and apply the first $T_C$ hours of the solution. We then move forward in time by $T_C$, compute a solution to the subsequent, shortened split-horizon problem, and so forth. Note that for the problem considered here, solutions to the LP are also feasible for the NLP, and thus the relative length of $T_C$ versus $T_1$ presents no issue for feasibility. 

When solving the split-horizon problem, the nonlinear first stage is also solved using the DPM toolbox. The linear second stage is solved using CPLEX 12.7.0. The simulations are run in \textsc{Matlab}, formulated by \textsc{YALMIP}, and solved on a six core Intel Core i7-5820K with 16 GB of RAM.

\begin{figure}[tp] 
	\begin{centering}
		\includegraphics[width=0.48\textwidth]{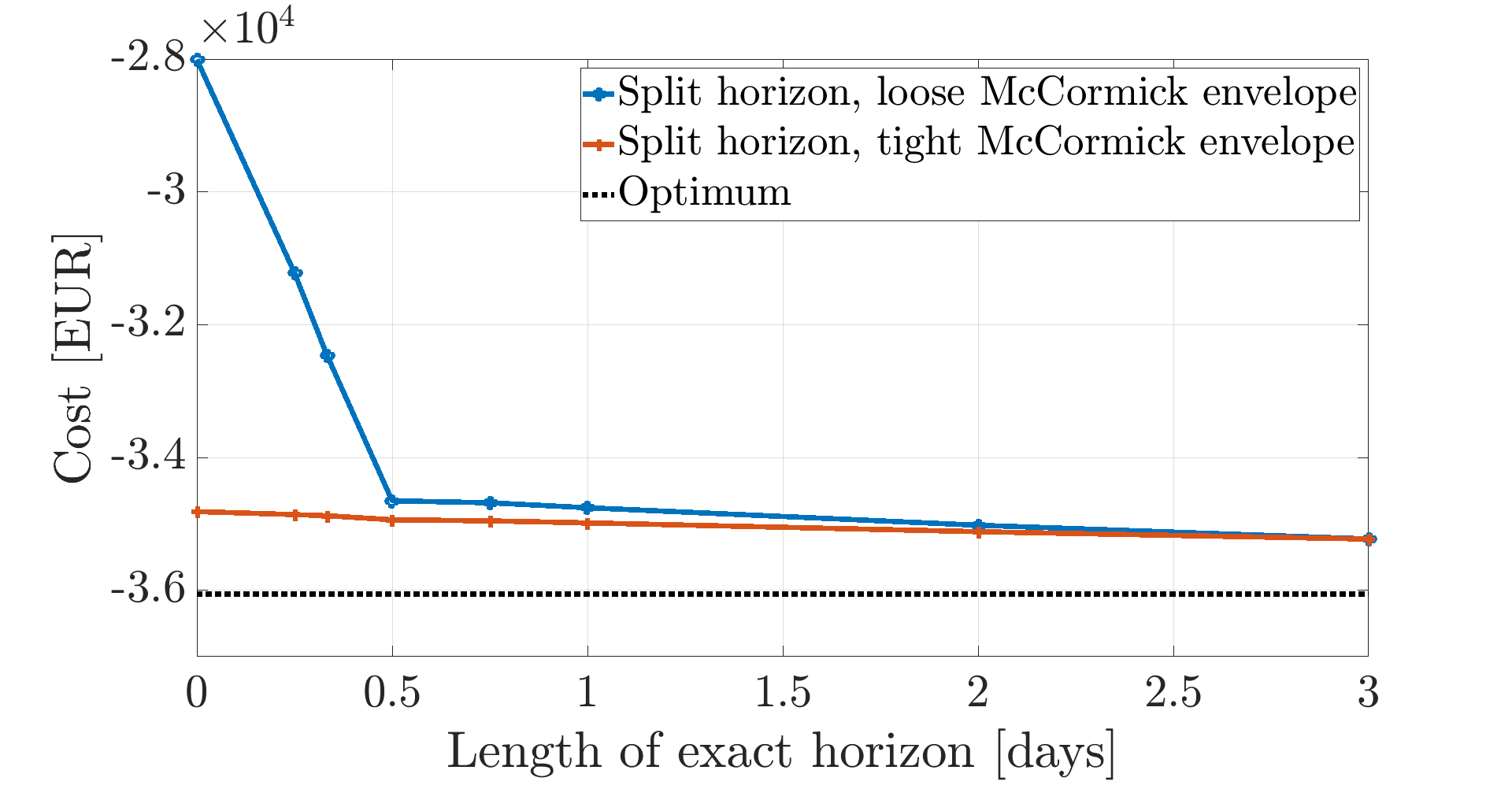} 
	\end{centering}
	\caption{Optimality of split-horizon DDP method over 20 days with varying length of exactly-modeled first stage, and control horizon of 12 hours. Split-horizon method uses McCormick envelopes in linear stage, with linearization limits either set to the given reservoir level limits (loose) or as in \eqref{eq:improving_mccormick_bounds_1} and \eqref{eq:improving_mccormick_bounds_2} (tight). Method applied to two-reservoir system of Figure \ref{fig:test_system_diagram}.}
	\label{fig:optimality_comparison_loose}
\end{figure}

In Figure \ref{fig:optimality_comparison_loose}, the two-stage problem is solved over various exact horizon lengths $T_1$. The figure shows that as $T_1$ increases (for a fixed total horizon), the cost decreases. When the McCormick bounds are set to the reservoir level limits, using the split-horizon method with $T_1 = 12$ hours results in a problem objective which is 23.7\% less than that from the linearized problem (where $T_1=0$ in the figure), and is within 3.9\% of the optimum. Since the suboptimality decreases with the exact horizon length, we would choose the longest $T_1$ that allows Algorithm \ref{alg:mpc_ddp_algorithm} to converge in the allotted time. For the hydro application here, the allotted time would typically be on the order of one hour, as in \cite{Abgottspon2016}.

\begin{figure}[tp] 
	\begin{centering}
		\includegraphics[width=0.48\textwidth]{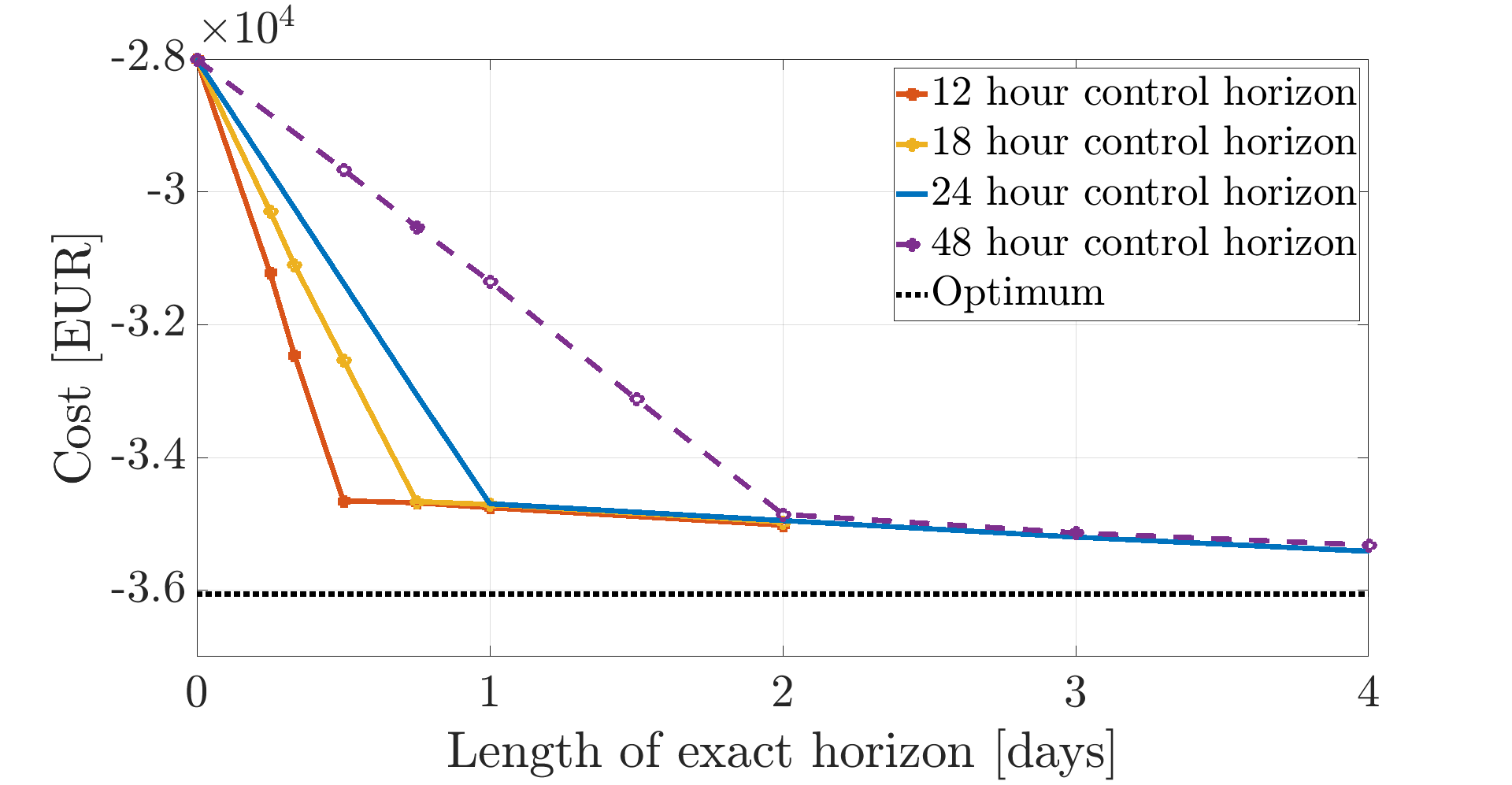} 
	\end{centering}
	\caption{Optimality of split-horizon DDP method over 20 days with varying length of exactly-modeled period. Results are for two-reservoir system. Control horizon of MPC decisions is varied.}
	\label{fig:optimality_comparison_loose_varied_decision_rate}
\end{figure}

When using the tightened McCormick bounds \eqref{eq:improving_mccormick_bounds_1} and \eqref{eq:improving_mccormick_bounds_2}, there is a 0.4\% improvement due to modeling 12 hours exactly. Thus, the utility of the split-horizon method depends on the level of accuracy in the linearization. Note that here, the improvement in linearization due to the tightened McCormick bounds is reduced as the length of the exact horizon increases.

The theoretical bounds from \eqref{eq:mccormick_approximation_error_bound} state that for the problem considered in Figure \ref{fig:optimality_comparison_loose}, if the two-stage problem with $T_1 = 12$ hours is solved to global optimality, the objective found for the first timestep is within $61.5\%$ of the optimum for the case of the loose McCormick bounds, and within $62.6\%$ of the optimum for the case of the tightened bounds. These bounds are conservative compared to the suboptimality achieved in simulation for two reasons. First, the McCormick envelope approximation considers the worst case combination of inputs, which rarely occurs in practice. This conservatism accumulates at each timestep, rendering bounds at more distant timesteps even more conservative. Second, the bounds are given for a single problem instance, and do not take into account that only part of each solution is used in the receding horizon setting.

In Figure \ref{fig:optimality_comparison_loose_varied_decision_rate}, we vary the control horizon $T_C$ over which the computed action is applied. As $T_1$ is also varied, a ``knee'' appears in the graph, coinciding with $T_C$. Since $T_C$ is usually a fixed problem parameter, this result suggests we should choose $T_1 \geq T_C$. The objective continues to improve as $T_1$ is increased, but not as significantly as before the ``knee." Also, as $T_C$ gets shorter, the objective improves. This is expected from the receding horizon context, as measurements from the true system are incorporated at a more frequent rate.

\begin{figure}[tp] 
	\begin{centering}
		\includegraphics[width=0.48\textwidth]{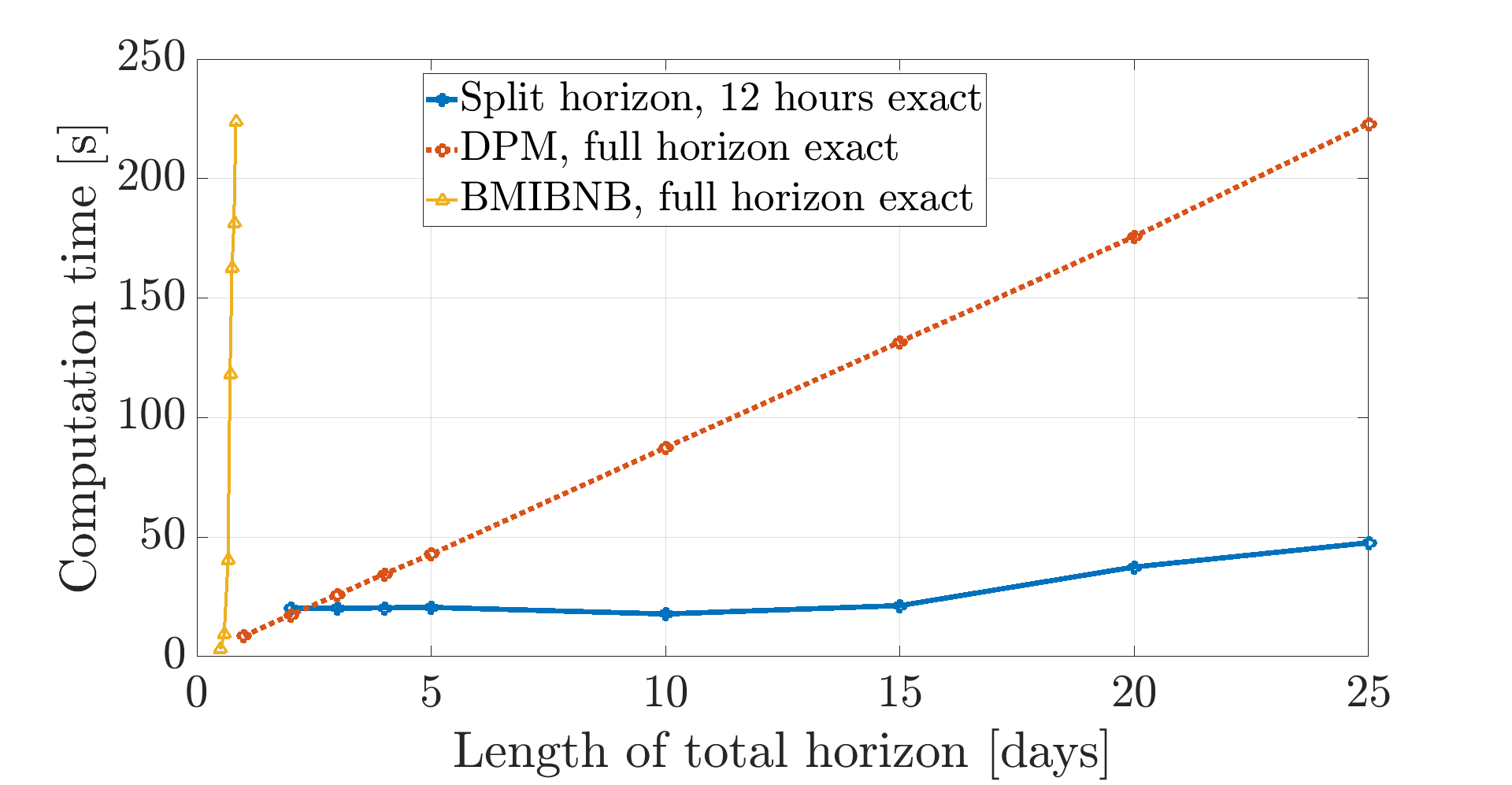} 
	\end{centering}
	\caption{Computation time for two-reservoir system of split-horizon DDP method with 12 hours exactly modeled, compared to solving full bilinear problem approximately using DPM dynamic programming toolbox and exactly using BMIBNB global solver.}
	\label{fig:computation_comparison}
\end{figure}

In Figure~\ref{fig:computation_comparison}, we compare the computation time for the split-horizon approximation  \eqref{eq:original_two_stage_problem} versus solving the exact bilinear problem \eqref{eq:bilinear_problem} using the DPM toolbox and BMIBNB solver. The split-horizon problem computation time is for solving the first instance of the problem i.e., we compute the solution for the full horizon once, and do not account for solving the problem repeatedly over the shrinking horizon. This reflects the envisioned receding horizon setting: in practice we would solve the problem once, use the solution for the first timestep, and then re-solve a new problem that incorporates updated problem data for the next timestep. We see that the split-horizon method with $T_1 = 12$ hours scales well with respect to the total horizon length. In practice, the increase in solution time depends mainly on that of the second stage LP, as the number of Benders cuts for the second stage tends to increase sublinearly as the second stage length grows. For example, the split horizon method in Figure~\ref{fig:computation_comparison} results in two cuts for total horizons less than five days, between three and four cuts from horizons from 5-19 days, and between four and five cuts for horizons up to 25 days. As expected theoretically, the computation time for solving the exact bilinear problem scales exponentially in horizon length when using the global BMIBNB solver, and scales linearly when solving approximately using the DPM toolbox. 

\subsection{Solution of Problems with Higher State Dimension}
For the two-reservoir system considered above, the DPM toolbox can solve the exact problem to within the tolerance of the given grid. However, the  chief disadvantage of dynamic programming is that the computation time scales exponentially in the problem dimension. When we modify the system of Figure~\ref{fig:test_system_diagram} to include an additional reservoir, increasing the state and input dimensions each to three, the DPM toolbox fails to return a solution due to reaching computer-specific RAM usage limits. It is possible to make the DP discretization coarser, but this comes at a cost to optimality. 

\begin{figure}[tp] 
	\begin{centering}
		\includegraphics[width=0.48\textwidth]{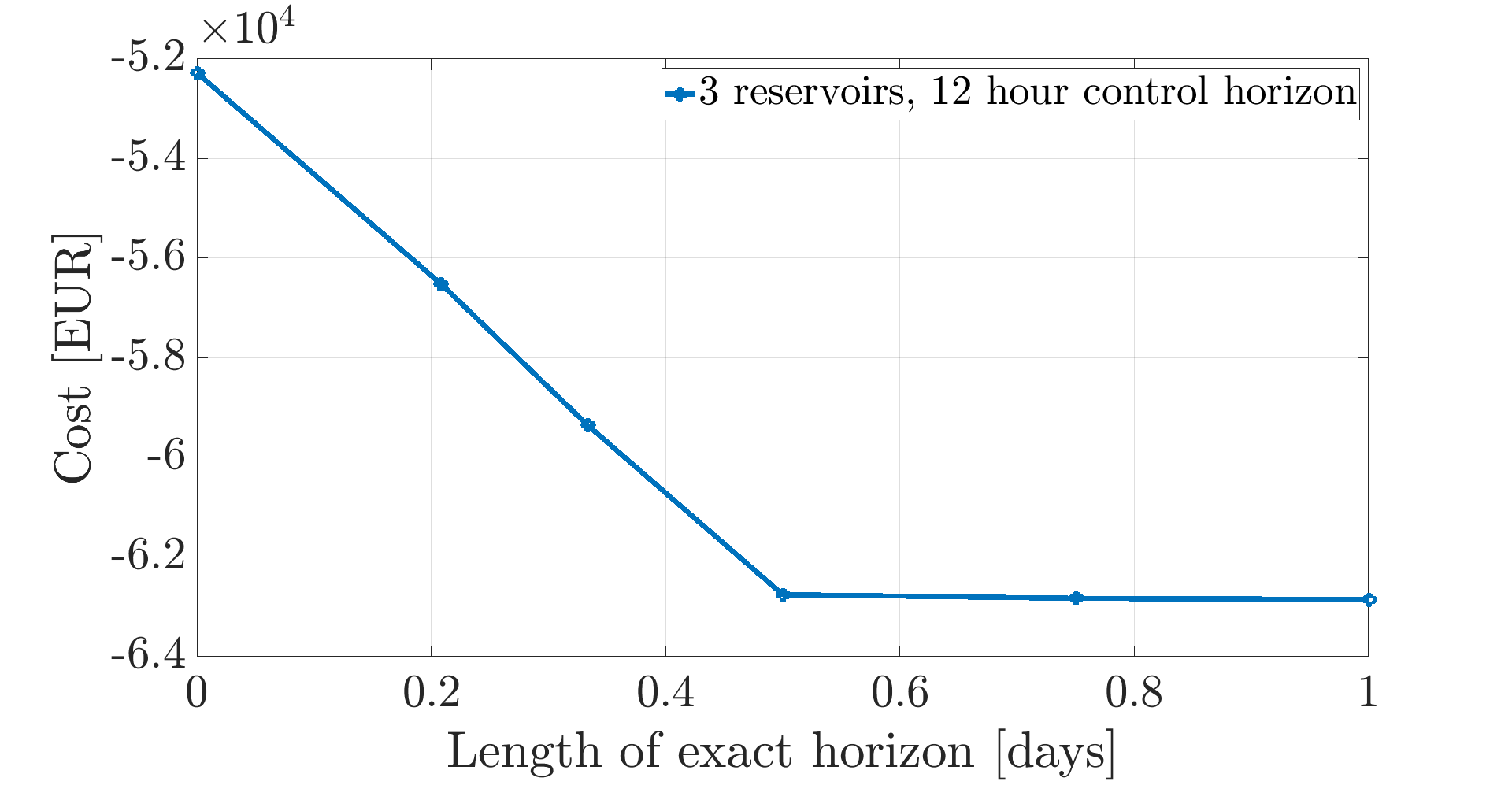} 
	\end{centering}
	\caption{Optimality of split-horizon DDP method over 20 days with varying length of exactly-modeled period. 12 hour control horizon used throughout. Results are for three reservoir system, where dimension of system is too large for DPM toolbox to solve.}
	\label{fig:optimality_comparison_loose_4_day}
\end{figure}

In contrast to dynamic programming, where one must solve over the entire state space, local solution methods can be used to solve the first-stage nonlinear problem of the split-horizon approximation. The length of the first stage can be chosen short enough so that it is computationally tractable to solve. The experimental results presented here suggest that using even short first stages for the split-horizon approximation leads to near-optimal solutions in a receding horizon setting. In Figure \ref{fig:optimality_comparison_loose_4_day}, we display the results of using the split-horizon method on the three reservoir system mentioned in the previous paragraph. Using the BMIBNB global solver (solved to a relative tolerance of $1 \times 10^{-6}$) with a control horizon of 12 hours, we found that the split-horizon method with an exact horizon of 12 hours improved the objective by 20\%, relative to the solution of the linearized problem. Note that solving the 20 day exact NLP to optimality is computationally intractable. 

\subsection{Comparison to Multi-cell Approximation} \label{section:multicell_comparison}
For a comparison with other methods which consider local approximations of the bilinearity (\cite{Diniz2008}-\cite{Cerisola2012}), we also implement the multi-cell McCormick envelope approximation method of \cite{Cerisola2012}. The method is used to approximate the bilinearity \eqref{eq:bilinear_constraint_power} over the entire horizon. Here, we choose to split each volume flow and reservoir level variable into two equal intervals, and generate a McCormick envelope for each of the four resulting cells. For the two-reservoir model of Figure \ref{fig:test_system_diagram}, this introduces one binary variable per timestep for each of the four devices and two reservoirs. Considering the 20 day horizon with hourly timesteps, this results in a mixed-integer linear program (MILP) with 2880 binary variables. We solve the MILP in the receding horizon setting of Figure \ref{fig:ddp_shrinking_horizon_method_illustration}, allotting a solution time equivalent to the time used by the split-horizon method for a given exact horizon length.

We produce Figure \ref{fig:optimality_comparison_multicell} by varying the exact horizon length of the split-horizon method. This can be viewed as exploration of a Pareto front between computation time and solution optimality. We see that the multi-cell McCormick method performs worse than the split-horizon method. For a 20 day horizon, when both methods use a computation time of 32.4 seconds (the solution time of the split-horizon method with an exact horizon of 0.5 days), the multi-cell method achieves an objective which is within 11.5\% of the optimum, while the split-horizon method is within 3.9\%. Moreover, we expect the computation time of the split-horizon method to scale better with increasing problem size, compared to the exponential complexity of solving a larger MILP.

\begin{figure}[tbp] 
	\begin{centering}
		\includegraphics[width=0.48\textwidth]{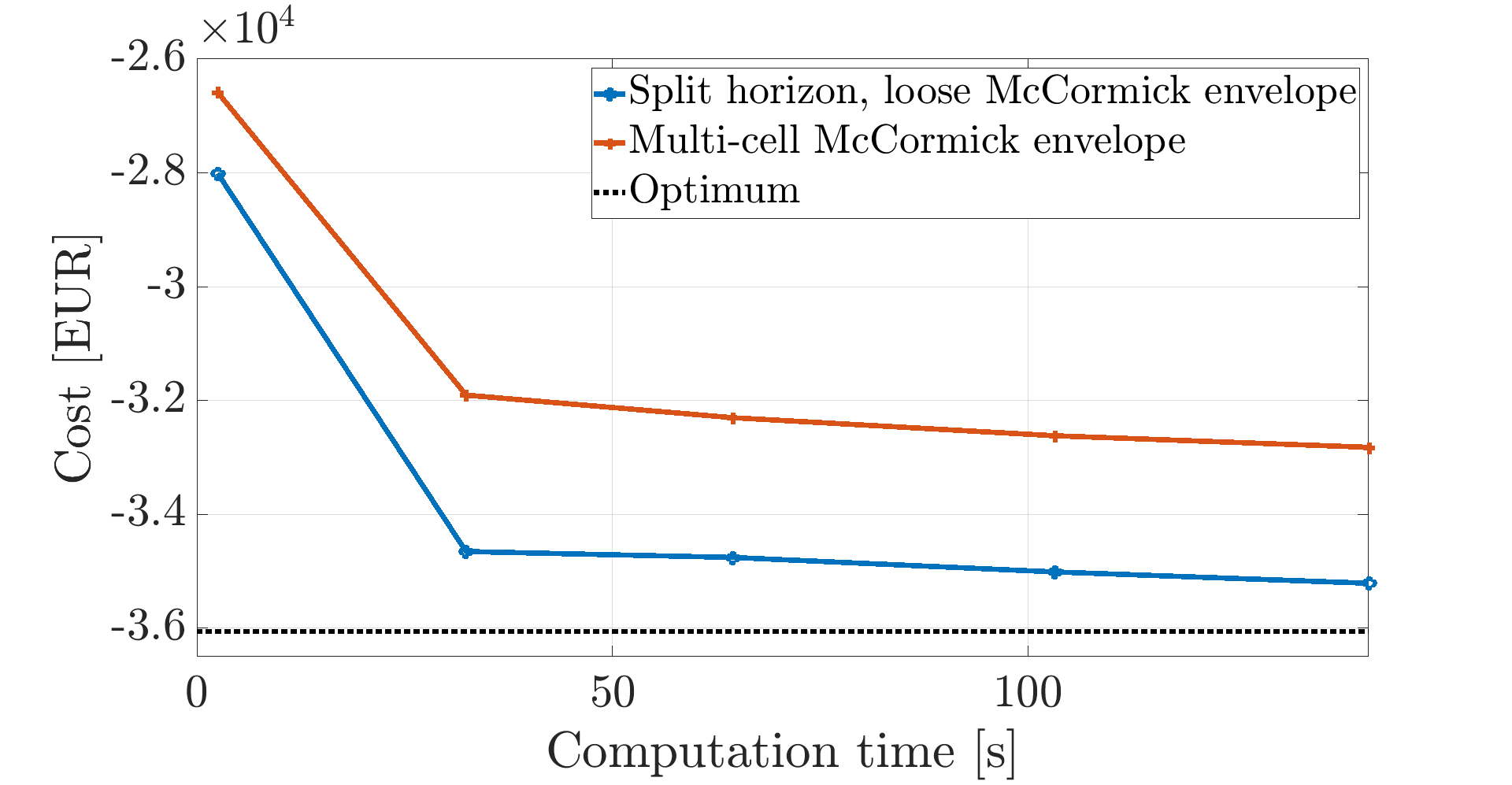} 
	\end{centering}
	\caption{Pareto curves for split-horizon DDP method versus multi-cell McCormick method \cite{Cerisola2012}, with volume flow and reservoir level variables each split into two intervals. Data points result from varying length of exactly-modeled first stage in DDP method, taking computation time, and setting this as maximum computation time for multi-cell method. Displayed is performance over 20 days with 12 hour control horizon, for two-reservoir system of Fig.~\ref{fig:test_system_diagram}.}
	\label{fig:optimality_comparison_multicell}
\end{figure}

\section{Conclusion}
We have shown that the DDP-based Algorithm \ref{alg:mpc_ddp_algorithm}, when applied to a split-horizon approximation of a bilinear hydro optimization problem, achieves accurate and computationally efficient results compared to solving either the full exact nonlinear problem or its linearization. When the optimization is conducted in a receding horizon manner with exact horizons on the order of the control horizon, the performance is nearly optimal, with significant computational savings when modeling higher-dimensional systems. 

The experimental results presented here demonstrate that accurately modeling the first part of the horizon can significantly improve the result relative to a linearized model. This improvement is due to the mismatch between the linearized and true model. Although better linearizations reduce the need for the split-horizon method, it might still be advantageous to use DDP with local linearizations like \cite{Cerisola2012} in the second stage. However, additional research along the lines of \cite{Zou2018} is needed to use DDP with the resulting integer variables. We note that the success of the split-horizon method depends on the type of nonlinearity considered, and how well applicable solution methods scale as the first stage length increases. 

We also derived bounds for the suboptimality of the split-horizon method. These bounds were found to be quite conservative for the problem here. Error bounds can be tighter for other problem cases and classes. For example, the suboptimality relative to an integer relaxation of an MILP can also be bounded, with a tightness that depends on the problem parameters. However, our case study here shows that while the bounds on a McCormick envelope approximation of a bilinearity are conservative, the split-horizon method used in a receding horizon setting can still achieve good results. Future research could explore how these bounds differ for a receding horizon setting, compared to the single-period optimization for which they were derived. 

The split-horizon method is well suited to be extended to stochastic problems e.g., for uncertain prices here. As noted above, stochasticity, perhaps incorporated using a scenario tree approach as in \cite{Rebennack2016}, further increases the complexity of solving the problem exactly over long horizons.

%

\section*{Acknowledgment}

The authors would like to thank Roy Smith and members of the Building Control Group at the ETH Automatic Control Laboratory for their input into this work. This work is supported by the SCCER FEEB\&D project and the European Research Council under the project OCAL, ERC-2017-ADG-787845.

\ifCLASSOPTIONcaptionsoff
  \newpage
\fi


\begin{thebibliography}{21}
%

\bibitem{USDoe2016}
U.S. Department of Energy, Office of Energy Efficiency and Renewable Energy, ``2016 Renewable Energy Data Book,'' https://www.nrel.gov/docs/fy18osti/70231.pdf. Accessed Dec. 1, 2017.

\bibitem{Chen2009}
H. S. Chen, T. N. Cong, W. Yang, C. Q. Tan, Y. L. Li, and Y.L. Ding, ``Progress in electrical energy storage system: A critical review,'' \textit{Progress in Natural Science}, vol. 19, no. 3, pp. 291–312, 2009.

\bibitem{Abgottspon2016}
H. Abgottspon and G. Andersson, ``Multi-horizon Modeling in Hydro Power Planning," \textit{Energy Procedia}, vol. 87, pp. 2–10, 2016.

\bibitem{Pritchard2005} 
G. Pritchard, A. B. Philpott, and P.J. Neame, ``Hydroelectric reservoir optimization in a pool market." \textit{Mathematical Programming}, vol. 103, no. 3, pp. 445–461, 2005.

\bibitem{Rotting1992}
T. A. Rotting and A. Gjelsvik, ``Stochastic dual dynamic programming for seasonal scheduling in the Norwegian power system," \textit{IEEE Transactions on Power Systems}, vol. 7, no. 1, pp. 273-279, Feb. 1992.

\bibitem{Diniz2008}
A. Diniz and M. Maceira, ``A Four-Dimensional Model of Hydro Generation for the Short-Term Hydrothermal Dispatch Problem Considering Head and Spillage Effects," \textit{IEEE Transactions on Power Systems}, vol. 23, no. 3, pp. 1298-1308, 2008. 

\bibitem{Borghetti2008}
A. Borghetti, C. D'Ambrosio, A. Lodi, and S. Martello, ``An MILP Approach for Short-Term Hydro Scheduling and Unit Commitment With Head-Dependent Reservoir'' \textit{IEEE Transactions on Power Systems}, vol. 23, no. 3, pp. 1115-1124, Aug. 2008.

\bibitem{Cerisola2012}
S. Cerisola, J.M. Latorre, and A. Ramos, ``Stochastic dual dynamic programming applied to nonconvex hydrothermal models," \textit{European Journal of Operational Research}, vol. 218, no. 3, pp. 687–697, 2012.

\bibitem{Catalao2009}
J. P. S. Catalao, S. J. P. S. Mariano, V. M. F. Mendes, and L. A. F. M. Ferreira, ``Scheduling of Head-Sensitive Cascaded Hydro Systems: A Nonlinear Approach," \textit{IEEE Transactions on Power Systems}, vol. 24, no. 1, pp. 337-346, Feb. 2009.

\bibitem{Pereira1991}
M. V. F. Pereira and L. M. V. G. Pinto. ``Multi-stage stochastic optimization applied to energy planning," \textit{Mathematical Programming}, vol. 52, pp. 359–375, 1991.

\bibitem{Grossmann1991}
N.V. Sahinidis and I.E. Grossmann, ``Convergence Properties of Generalized Benders Decomposition,'' \textit{Computers and Chemical Engineering}, vol. 15, no. 7, pp. 481-491, 1991.

\bibitem{Zou2018}
J. Zou, S. Ahmed, and X. A. Sun, ``Stochastic dual dynamic integer programming,'' \textit{Mathematical Programming}, vol. 175, no. 1-2, pp. 461-502, 2019.

\bibitem{AbgottsponThesis2015} 
H. Abgottspon, ``Hydro power planning: Multi-horizon modeling and its applications,'' ETH Zürich PhD Dissertation No. 22729, 2015.

\bibitem{Hohmann2018}
M. Hohmann, J. Warrington, and J. Lygeros, ``A Moment and Sum-of-Squares Extension of Dual Dynamic Programming with Application to Nonlinear Energy Storage Problems,'' \textit{European Journal of Operational Research}, 2019.

\bibitem{Flamm2018}
B. Flamm, A. Eichler, J. Warrington, and J. Lygeros, ``Dual Dynamic Programming for Nonlinear Control Problems over Long Horizons,'' \textit{2018 European Control Conference (ECC)}, Limassol, 2018, pp. 471-476.

\bibitem{Guigues2018}
V. Guigues, ``Inexact Cuts in Deterministic and Stochastic Dual Dynamic Programming Applied to Linear Optimization Problems,'' 2018. https://arxiv.org/abs/1801.04243.

\bibitem{Rebennack2016}
S. Rebennack. ``Combining sampling-based and scenario-based nested Benders decomposition methods: application to stochastic dual dynamic programming," \textit{Mathematical Programming}, vol. 156, no. 1–2, pp. 343–389, 2016.

\bibitem{Rahmaniani2017}
R. Rahmaniani, T. G. Crainic, M. Gendreau, and W. Rei, ``The Benders decomposition algorithm: A literature review,'' \textit{European Journal of Operational Research}, vol. 259, no. 3, pp. 801-817, 2017.

\bibitem{Lofberg2004}
J. Löfberg, ``YALMIP: a toolbox for modeling and optimization in MATLAB," \textit{2004 IEEE International Conference on Robotics and Automation}, pp. 284-289, 2004.

\bibitem{SwissSpotPrice}
European Power Exchange Swiss Spot Market Price. www.epexspot.com/en. Accessed Oct. 25, 2017.

\bibitem{Sundstrom2009}
O. Sundstrom  and L. Guzzella, ``A generic dynamic programming Matlab function,'' In \textit{Proceedings of the 18th IEEE International Conference on Control Applications}, pp. 1625-1630, Saint Petersburg, Russia, 2009.



\end{thebibliography}
\end{document}